\documentclass[12pt]{amsart}
\usepackage{amsmath,amssymb,amsfonts,amsthm}
\usepackage{xcolor}
\usepackage{enumerate}
\usepackage[hyphens]{url}
\usepackage[utf8]{inputenc}
\usepackage{mathtools}   % loads »amsmath
\usepackage[normalem]{ulem}% use for \sout
\usepackage[color]{changebar}
\usepackage[margin=1in]{geometry}

\usepackage{ mathrsfs } %fancy cursive

\theoremstyle{definition} %Tira o itálico dos teoremas

\newtheorem{thm}{T}%[chapter]
\numberwithin{thm}{section}  %poderia ser subsection
\numberwithin{equation}{section}

\newtheorem{definition}[thm]{Definition}
\newtheorem{theorem}[thm]{Theorem}
\newtheorem{lemma}[thm]{Lemma}
\newtheorem{corollary}[thm]{Corollary}
\newtheorem{conjecture}[thm]{Conjecture}
\newtheorem{proposition}[thm]{Proposition}
\newtheorem{remark}[thm]{Remark}
\newtheorem{example}[thm]{Example}

\newtheorem*{definition*}{Definition}

\newtheorem{innercustomthm}{Theorem}
\newenvironment{customthm}[1]
  {\renewcommand\theinnercustomthm{#1}\innercustomthm}
  {\endinnercustomthm}

\usepackage{tikz}
\usepackage{tikz-cd} %commutative diagrams
\usetikzlibrary{arrows}
\usetikzlibrary{arrows.meta}

\newcommand*\circled[1]{\tikz[baseline=(char.base)]{
   \node[shape=circle,draw,inner sep=1pt, thick] (char) {#1};}}
\newcommand*\circledRed[1]{\text{\normalsize $\color{rgb,255:red,180;green,51;blue,51}\circled{#1}$}}

\begin{document}

\title{Positivity of Matui's HK Conjecture for AF Groupoids}

\author{}
\author{Rafael P. Lima}

\address[R.P.\ Lima]{School of Mathematics and Statistics, Victoria University of Wellington, PO Box 600, Wellington 6140, NEW ZEALAND}
\email{{lima.rafael.p@gmail.com}}

\subjclass[2020]{22A22(primary), 46L05 (secondary)}
\keywords{elementary groupoid, AF groupoid}

\thanks{This research was funded by the Marsden Fund of the Royal Society of New Zealand (grant number 18-VUW-056)}

\begin{abstract}
In this paper we generalise an application of Matui's HK conjecture by Farsi, Kumjian, Pask, and Sims, that gives an isomorphism from the homology groups of AF groupoids to the corresponding K-theory. We give an explicit formula for this isomorphism, and we show that the map is an order isomorphism. Since homology groups are equipped with several useful techniques, this map can help us to understand the K-theory of the C*-algebra in more detail. To illustrate this, we apply the isomorphism to characterise the AF embeddability of the C*-algebra of Deaconu-Renault groupoids. %, finding an alternative proof to (Schafhauser, JFA, 2015)
\end{abstract}

\maketitle

\section{Introduction}

Matui \cite{Matui} studied how homology groups of groupoids can describe the dynamical properties the topological full groups of a groupoid. There, the homology groups $H_0(G)$ and $H_1(G)$ for \'etale groupoids over totally disconnected spaces are described by adapting the construction from Crainic and Moerdijk \cite{Crainic}. In \cite[Conjecture 2.6]{matui2015topological}, Matui formulated the HK conjecture, which connects the homology groups with the K-groups of the associated C*-algebras, stated below:

\begin{conjecture}[Matui's HK conjecture]
    Let $G$ be a locally compact, Hausdorff, \'etale groupoid, such that $G^{(0)}$ is a Cantor set. Suppose that $G$ is both effective and minimal. Then for $j=0,1$ we have
    \[
    K_j(C^*_r(G)) \cong \bigoplus_{i=0}^\infty H_{2i + j}(G).
    \]
\end{conjecture}

This conjecture holds for AF groupoids with compact unit spaces \cite[Theorem 4.10]{Matui}, and for \'etale groupoids arising from subshifts of finite type \cite{Matui}. Other examples can be found in \cite{bonicke2023dynamic, nyland2021matui, ortega2020homology, ortega2023almost, ProiettiYamashitaI, tanner2023rigidity, yi2020homology}. Proietti and Yamashita \cite{ProiettiYamashitaIV} found an example beyond the original assumptions of the conjecture, while Scarparo \cite{scarparo2020homology}, Deeley \cite{Deeley2023counterexample}, and Ortega and Sanchez \cite{OrtegaSanchez} provide counterexamples.

Farsi, Kumjian, Pask, and Sims \cite{FKPS} proved Matui's HK Conjecture for Deaconu-Renault groupoids of ranks one and two, and for k-graph groupoids of single vertex k-graphs which satisfy a mild joint-coprimality condition. Moreover, they generalised Matui's \cite[Theorem 4.10]{Matui}, showing in \cite[Corollary 5.2]{FKPS} that the conjecture holds for AF groupoids with non-compact unit spaces.

Additionally, Farsi et al. show that, for an AF groupoid $G$, there exists an isomorphism between the K-group $K_0(C^*(G))$ and the corresponding homology group $H_0(G)$. This map is such that $[1_V]_0 \mapsto [1_V]$ for every compact open set $V \subset G^{(0)}$. However, it is not clear from their paper whether this map is an isomorphism of ordered groups, i.e., that it preservers positive maps from one group to the other.

In this paper, we show that the map provided in Farsi et al. for AF groupoids is an ordered group isomorphism. Moreover, we give an explicit formula for the isomorphism, as we state below.

\begin{customthm}{1}
\label{thm:HK}
Let $G$ be an AF groupoid, written as the inductive limit $G = \displaystyle\lim_\rightarrow G_n$ of elementary groupoids. Then, for every $n$, there exists a locally compact, Hausdorff, second countable, totally disconnected space $Y_n$, and there is a surjective local homeomorphism $\sigma_n: G_n^{(0)} \rightarrow Y_n$ such that $G_n = R(\sigma_n)$. Moreover,
\begin{enumerate}[1.]
\item there exists a unique isomorphism $\nu: K_0(C^*(G)) \rightarrow H_0(G)$ defined by
\begin{align*}
\nu([p]_0) = [\mathrm{tr}_{\varphi_n}(p)]
\end{align*}
for $p \in \mathcal{P}_n(C^*(G_n))$, and such that $\varphi_n$ is an arbitrary continuous section of $\sigma_n$;
\item the map $\nu$ is an isomorphism of ordered groups, i.e., $\nu(K_0(C^*(G))^+) = H_0(G)^+$; and
\item the isomorphism $\nu$ is precisely the map given by \cite[Corollary 5.2]{FKPS}. In other words, $\nu([1_V]_0) = [1_V]$ for all $V \subset G^{(0)}$ compact open.
\end{enumerate}
\end{customthm}

The theorem presents two advantages. First, K-groups are important in the theory of C*-algebras, since several results rely on them. However, K-theory is often too abstract and hard to calculate. In contrast, homology groups are more manageable and are equipped with techniques that make them easier to understand. So, for the C*-algebra of an AF groupoid, the explicit formula above can be applied to understand the K-theory in more detail. Second, in some situations, we need to know the positive elements of the $K_0$-group, so an application of Matui's HK conjecture would be useless unless we guarantee that the isomorphism preserves positive elements from one group to another.

Note that Theorem \ref{thm:HK} brings another question: for other groupoids where Matui's HK conjecture hold, is the group isomorphism $K_0(C^*_r(G)) \rightarrow \bigoplus_{i=0}^\infty H_{2i}(G)$ also an ordered group isomorphism?

In the second part of the paper, we apply Theorem \ref{thm:HK} to solve a more practical problem, characterising when the C*-algebra of a large class of Deaconu-Renault groupoids is AF embeddable. The example illustrates the power of homology techniques when one needs to manipulate the elements of the $K_0$-group. We state the theorem as follows:

\begin{definition*}
Let $\sigma: X \rightarrow Y$ be a surjective local homeomorphism on locally compact, Hausdorff, second countable, totally disconnected spaces, and let $\varphi: Y \rightarrow X$ be a continuous section of $\sigma$. We define the map $\mathrm{tr}_\varphi: \mathcal{P}_\infty(C^*(R(\sigma))) \rightarrow C_c(X, \mathbb{N})$ by
\begin{itemize}
\item $\mathrm{tr}_\varphi(p)(x) = 1_{X_\varphi}(x) \displaystyle\sum_{y: \sigma(x) = \sigma(y)} p(y)$, \hspace{15pt} for $p \in C^*(R(\sigma))$, $x \in X$, and
\item $\mathrm{tr}_\varphi(p)(x) = \displaystyle\sum_{i=1}^n \mathrm{tr}_\varphi(p_{ii})(x)$, \hspace{52pt} for $p \in \mathcal{P}_n(C^*(R(\sigma)))$, $n \geq 1$,
\end{itemize}
where we identify $C^*(R(\sigma))$ as a subset of $C_0(R(\sigma))$ by applying the $j$ map from \cite[Proposition II.4.2]{Renault} and \cite[Proposition 9.3.3]{sims2017hausdorff}.
\end{definition*}

\begin{customthm}{2}
\label{thm:main}
Let $\sigma$ be a surjective local homeomorphism on a locally compact, Hausdorff, second countable, totally disconnected space $X$. Denote by $\mathcal{G}$ be the Deaconu-Renault groupoid corresponding to $\sigma$. Then the following are equivalent:
\begin{enumerate} [(i)]
\item $C^*(\mathcal{G})$ is AF embeddable,
\item $C^*(\mathcal{G})$ is quasidiagonal,
\item $C^*(\mathcal{G})$ is stably finite,
\item $\mathrm{Im}(\sigma_\ast - \mathrm{id}) \cap C_c(X, \mathbb{N}) = \lbrace 0 \rbrace$,
\end{enumerate}
where the map $\sigma_\ast: C_c(X, \mathbb{Z}) \rightarrow C_c(X, \mathbb{Z})$ is defined by
\begin{align*}
\sigma_\ast(f)(x) = \sum_{y: \sigma(y) = x} f(y)
\hspace{20pt}
\text{for $x \in X$, $f \in C_c(X, \mathbb{Z})$.}
\end{align*}
\end{customthm}

The construction of the Deaconu-Renault groupoid $\mathcal{G}$ depends on the surjective local homeomorphism $\sigma$. The contribution of Theorem \ref{thm:main} is that it provides a condition on $\sigma$ that is equivalent to the AF embeddability of the C*-algebra of the corresponding Deaconu-Renault groupoid. This theorem is analogous to similar results for graph algebras \cite{Schafhauser-AFEgraph}, for crossed products of unital commutative C*-algebras by the integers \cite[Theorem 11.5]{Brown-quasidiagonal}, \cite[Theorem 9]{Pimsner}. Schafhauser \cite[Theorem C]{SchafhauserJFA2015} studied Cuntz-Pimsner algebras. Kumjian and Pask \cite{KP-kgraph} study higher rank graph algebras and give sufficient conditions that make the C*-algebras AF or purely infinite. Clark, an Huef, and Sims \cite{CaHS} characterised the AF embeddability of C*-algebras of row-finite and cofinal $k$-graphs with now sources, for $k = 2$. Schafhauser \cite[Corollary 6.2]{Schafhauser-AFEsimple} extendded this result for an arbitrary $k$. In particular, we apply the same strategy from \cite{CaHS}, by applying \cite[Theorem 0.2]{Brown-AFE} to describe AF embeddability in terms of a K-theory condition, and then analysing this condition to find a more concrete condition which can be described by the equation (iv) of the Theorem \ref{thm:main}.

Note that C*-algebras of Deaconu-Renault groupoids generalise graph algebras \cite{KPRR} and are examples of topological graph algebras \cite[Proposition 10.9]{KatsuraII}. When $X$ is compact, then the corresponding topological graph algebra is unital and, in this case, Theorem \ref{thm:main} is an example of \cite[Theorem 6.7]{Schafhauser-topologicalgraphs}, that characterises when a unital topological graph algebra is AF embeddable.

%In Subsection \ref{section:preliminaries:DRgroupoid}, we describe this class of groupoids in more detail.

In order to prove Theorem \ref{thm:main}, we apply the following result by Brown \cite{Brown-AFE}, that gives an equivalent condition to AF embeddability of crossed products of AF algebras by the integers:

\begin{definition*}
If $B$ is an AF algebra and $\beta \in \mathrm{Aut}(B)$, then we denote by $H_\beta$ the subgroup of $K_0(B)$ given by all elements of the form $K_0(\beta)(x) - x$ for $x \in K_0(B)$.
\end{definition*}

\begin{theorem} (Brown) \cite[Theorem 0.2]{Brown-AFE}
\label{thm:brown}
If $B$ is an AF algebra and $\beta \in \mathrm{Aut}(B)$, then the following are equivalent:
\begin{enumerate} [(a)]
\item $B \rtimes_\beta \mathbb{Z}$ is AF embeddable,
\item $B \rtimes_\beta \mathbb{Z}$ is quasidiagonal,
\item $B \rtimes_\beta \mathbb{Z}$ is stably finite,
\item $H_\beta \cap K_0(B)^+ = \lbrace 0 \rbrace$.
\end{enumerate}
\end{theorem}

Note that AF embeddability, quasidiagonality and stable finiteness are preserved by stable isomorphisms. If we choose an AF algebra $B$ such that $B \rtimes_\beta \mathbb{Z}$ is stably isomorphic to $C^*(\mathcal{G})$, then we can apply Brown's theorem to get a characterisation of the AF embeddability of $C^*(\mathcal{G})$. Indeed, we let $c: \mathcal{G} \rightarrow \mathbb{Z}$ be the canonical cocycle, and we let $B = C^*(\mathcal{G}(c))$ be the C*-algebra of the skew product $\mathcal{G}(c)$. This C*-algebra is AF by \cite[Lemma 6.1]{FKPS}. Theorem III.5.7 of \cite{Renault} gives the isomorphism
\begin{align*}
B \cong C^*(\mathcal{G}) \rtimes_{\alpha^c} \mathbb{T},
\end{align*}
where $\alpha^c$ is the action of $\mathbb{T}$ on $C^*(\mathcal{G})$ corresponding to the cocycle $c$. We let $\beta$ be an automorphism on $B$ corresponding to the dual map $\widehat{\alpha}^c$ with respect to the isomorphism above. By applying the crossed products on both sides of the equivalence above with the respective maps, we get the isomorphism $B \rtimes_\beta \mathbb{Z} \cong C^*(\mathcal{G}) \rtimes_{\alpha^c} \mathbb{T} \rtimes_{\widehat{\alpha}^c} \mathbb{Z}$, and by Takai duality, we have that $B \rtimes_\beta \mathbb{Z}$ is stably isomorphic to $C^*(\mathcal{G})$. So, we apply Brown's theorem to $C^*(\mathcal{G})$ and we get the equivalence of items (i)--(iii) of Theorem \ref{thm:main} with item (d) of Brown's theorem.

This equivalence gives a characterisation of when the C*-algebra is AF embeddable. Note that item (d) depends on the K-theory, which is very abstract. Moreover, in order to understand (d), we need to know the positive elements of $K_0(B)$. We use the theory of homology groups of groupoids in order to show that (d) is equivalent to condition (iv) of Theorem \ref{thm:main}. By applying results from \cite{FKPS, Matui} on homology groups, we get the commutative diagram below.

\begin{equation}
\label{eqn:diagram}
 \centering
\begin{tikzcd}
      K_0(B) \arrow[r, "K_0(\beta)"] \arrow{d}[swap]{\cong} 
      & K_0(B) \arrow[d, "\cong"]\\
      H_0(\mathcal{G}(c))  \arrow[r, "{[\widetilde{\beta}]}_{\mathcal{G}(c)}"] \arrow{d}[swap]{\cong} & H_0(\mathcal{G}(c))  \arrow[d,"\cong"]\\
      H_0(c^{-1}(0))  \arrow{r}{[\sigma_\ast] }[swap]{} & H_0(c^{-1}(0))
\end{tikzcd}
\end{equation}

It follows from \cite[Corollary 5.2]{FKPS} that $\mathcal{G}(c)$ is AF groupoid. Then Theorem \ref{thm:HK} gives the ordered group isomorphism $K_0(B) \rightarrow H_0(\mathcal{G}(c))$.  The groupoid $c^{-1}(0)$ is the kernel of the cocycle $c:\mathcal{G} \rightarrow \mathbb{Z}$ defined by $c(x,k,y) = k$. We define the isomorphism $H_0(\rho_\varphi): H_0(\mathcal{G})(c)) \rightarrow H_0(c^{-1}(0))$ in Section \ref{section:diagram}. The homomorphisms $K_0(\beta)$ and $[ \widetilde{\beta} ]_{\mathcal{G}(c)}$ are group homomorphisms induced by $\beta$, and the group homomorphism $[\sigma_\ast]$ is induced by the map $\sigma_\ast$ of Theorem \ref{thm:main}.

By studying the diagram and by applying properties of homology groups of groupoids, we can show that items (d) and (iv) are equivalent, proving Theorem \ref{thm:main}.

%We need to know the positive elements of $K_0(B)$ in order to study item (d). So, when we apply the commutative diagram, we need to guarantee that the isomorphisms above takes positive elements from one group to the another. Although \cite[Corollary 5.2]{FKPS} gives the existence of the group isomorphism $K_0(B) \cong H_0(\mathcal{G}(c))$, we need this isomorphism to be an ordered group isomorphism. We prove that this isomorphism is indeed a ordered group isomorphism using different techniques, and we give an explicit formula for the map, stated as follows:

%This theorem also generalises Matui's \cite[Theorem 4.10]{Matui}, which also gave an order isomorphism for AF groupoids with compact unit spaces, but did not provide a formula for the isomorphism. So, we extend Matui's result to AF groupoids with non-compact unit spaces. Moreover, by showing that $H_0(G)^+ \cong K_0(C^*(G))^+$, Theorem \ref{thm:HK} suggests a stronger version of Matui's HK conjecture, which considered only group homomorphisms.

We organise the paper as follows: in Section \ref{section:preliminaries}, we study inductive limits, AF groupoids, abstract homology groups, and homology groups for groupoids. In Section \ref{section:homology}, we study the homology groups of AF groupoids, while in Section \ref{section:KtheoryAFgroupoids} we prove Theorem \ref{thm:HK}, our main result. In Section \ref{section:application}, we apply the isomorphism from Theorem \ref{thm:HK} and the tools from homology groups to characterise the AF embeddability of the C*-algebra of the Deaconu-Renault groupoid by Theorem \ref{thm:main}.  Finally, in Section 6, we show that this theorem generalises known results in the literature.

%\textcolor{red}{We organise the paper as follows: in Section \ref{section:preliminaries}, we define Deaconu-Renault groupoids, inductive limits and AF groupoids, presenting some results from the literature. In Section \ref{section:skewproducts}, we state the isomorphism by \cite[Proposition II.5.1]{Renault} between the C*-algebra of a skew-product groupoid and its corresponding crossed product. We also define the map $\beta$, used later in the commutative diagram. Then we study the theory of homology groups for groupoids in Section \ref{section:homology}. In Section \ref{section:KtheoryAFgroupoids}, we prove Theorem \ref{thm:HK}, giving an isomorphism between the homology groups for AF groupoids and the K-theory of the C*-algebras of these groupoids. Similar results have been proved in \cite[Corollary 5.2]{FKPS}, \cite[Theorem 4.10]{Matui}. Our innovation is that we give an explicit formula for this isomorphism, and we show that it preserves positive elements from one group to another. In Section \ref{section:diagram} we apply the results from Sections  \ref{section:homology} and \ref{section:KtheoryAFgroupoids} to prove the commutative diagram \eqref{eqn:diagram}. In Section \ref{section:main} we prove Theorem \ref{thm:main}, our main result, that describes under certain conditions when the C*-algebra of a Deaconu-Renault groupoid is AF embeddable. Finally, we show in Section \ref{section:examples} that our main theorem generalises some known results in the literature.}

The author would like to thank Astrid an Huef and Lisa Orloff Clark for their supervision, Camila F. Sehnem for the helpful discussions, and  Nathan Brownlowe, Carla Farsi, and Hung Pham for their suggestions.

\section{Preliminaries}
\label{section:preliminaries}

\subsection{Inductive limits}

An \emph{inductive sequence} in a category $\mathscr{C}$ is a sequence $\lbrace A_n \rbrace_{n = 1}^\infty$ of objects in $\mathscr{C}$ and a sequence $\varphi_n: A_n \rightarrow A_{n+1}$ of morphisms in $\mathscr{C}$, usually written as
\begin{equation}
\label{eqn:inductivesequence}
\begin{tikzcd}
A_1 \arrow{r}{\varphi_1}[swap]{} &
A_2 \arrow{r}{\varphi_2}[swap]{} &
A_3 \arrow{r}{\varphi_3}[swap]{} &
\cdots
\end{tikzcd}    
\end{equation}

For $m > n$ we shall also consider the composed morphisms
\begin{align*}
\varphi_{m,n} = \varphi_{m-1} \circ \varphi_{m-2} \circ \cdots \circ \varphi_n: A_n \rightarrow A_m,
\end{align*}
which, together with the morphisms $\varphi_n$, are called the \emph{connecting morphisms} (or \emph{connecting maps}). An \emph{inductive limit} of the inductive sequence \eqref{eqn:inductivesequence} in a category $\mathscr{C}$ is a system $(A, \lbrace \mu_n \rbrace_{n=1}^\infty)$, where $A$ is an object in $\mathscr{C}$, $\mu_n: A_n \rightarrow A$ is a morphism in $\mathscr{C}$ for each $n \geq 1$, and where the following two conditions hold.
\begin{enumerate}[(i)]
\item The diagram
$$
\begin{tikzcd}
A_n \arrow{rr}{\varphi_n}[swap]{} \arrow{dr}{}[swap]{\mu_n} &&
A_{n+1} \arrow{dl}{\mu_{n+1}}[swap]{}\\
 & A
\end{tikzcd}
$$
commutes for each positive integer $n$.

\item If $(B, \lbrace \lambda_n \rbrace_{n=1}^\infty)$ is a system, where $B$ is an object in $\mathscr{C}$, $\lambda_n: A_n \rightarrow B$ is a morphism in $\mathscr{C}$ for each $n$, and where $\lambda_n = \lambda_{n+1} \circ \varphi_n$ for all  $n \geq 1$, then there is one and only one morphism $\lambda: A \rightarrow B$ making the diagram
$$
\begin{tikzcd}
& A_n \arrow{dl}{}[swap]{\mu_n} \arrow{dr}{\lambda_n}[swap]{} \\
A \arrow{rr}{}[swap]{\lambda} && B
\end{tikzcd}
$$
commutative for each $n$.
\end{enumerate}
In this text, will call the morphisms $\mu_n$ the \emph{inclusion morphisms} of $A$

Inductive limits, when they exist, are essentially unique in the sense that if $(A, \lbrace \mu_n \rbrace)$ and $(B, \lbrace \lambda_n \rbrace)$ are inductive limits of the same inductive sequence, then there is an isomorphism $\lambda:A \rightarrow B$ making the diagram of item (ii) above commutative. See \cite[page 93]{RLL-Ktheory} for details. Because of the essential uniqueness of inductive limits (when they exist), we shall refer to \textit{the} inductive limit (rather than \textit{an} inductive limit).

Consider an AF algebra $A = \overline{\bigcup_{n=0}^\infty A_n}$, where each $A_n \subset A_{n+1}$ is finite dimensional. Then $A = \displaystyle\lim_{\longrightarrow} A_n$ and $K_0(A) = \displaystyle\lim_{\longrightarrow} K_0(A_n)$.

\begin{lemma}
\label{lemma:homomorphisminductive}
Let $A$ be an object in a category $\mathscr{C}$ such that $A$ is the inductive limit of an inductive sequence with connecting morphisms $i_n: A_n \rightarrow A_{n+1}$ and inclusion morphisms $\alpha_n: A_n \rightarrow A$. Similarly, let $B$ be an inductive limit in $\mathscr{C}$ with connecting morphisms $j_n: B_n \rightarrow B_{n+1}$ and inclusion morphisms $\beta_n: B_n \rightarrow B$. Let $\lbrace \mu_n \rbrace_{n = 0}^\infty$ be a sequence of morphisms $\mu_n : A_n \rightarrow B_n$ satisfying $\mu_{n+1} \circ i_n = j_n \circ \mu_n$. Then there exists a unique morphism $\mu: A \rightarrow B$ making the diagram below commutative.
\begin{equation}
\label{eqn:homomorphisminductive}
\begin{tikzcd}
A_n \arrow[r,"\mu_n"] 
\arrow{d}{}[swap]{\alpha_n}
& B_n \arrow[d, "\beta_n"]\\
A \arrow{r}{\mu}[swap]{}
& B
\end{tikzcd}
\end{equation}
Moreover, if each $\mu_n$ is an isomorphism, we have that $\mu$ is an isomorphism.
\end{lemma}

\subsection{AF groupoids}

There are different definitions for AF groupoids in the literature. The first was given by Renault \cite[Definition III.1.1]{Renault} in his book, and he proved in \cite[Proposition III.1.15]{Renault} that AF groupoids determine AF algebras up to isomorphism. Because of this result, we can see AF groupoids as groupoid analogues of AF algebras. Other definitions for such groupoids have been introduced in the literature \cite{Renault-AF}, \cite[Definition 3.7]{GPS}, \cite[Definition 2.2]{Matui}, and \cite[Definition 4.9]{FKPS}. In \cite{CaHLS} there is a proof of the equivalence of the definitions \cite[Definition III.1.1]{Renault} and \cite[Definition 4.9]{FKPS}. In this paper, we consider the Definition 4.9 of \cite{FKPS}, by Farsi, Kumjian, Pask, and Sims.

Given the topological spaces $X$ and $Y$, a map $\sigma: X \rightarrow Y$ is a \emph{local homeomorphism} if every $x \in X$ has an open neighbourhood $\mathcal{U}$ such that $\sigma(\mathcal{U})$ is open and $\sigma\vert_{\mathcal{U}}: \mathcal{U} \rightarrow \sigma(\mathcal{U})$ is a homeomorphism. A continuous map $\varphi: Y \rightarrow X$ is a continuous section of $\sigma$ if $y = \sigma(\varphi(y))$ for all $y$. Using the argument from the proof of \cite[Theorem 4.10]{FKPS}, one can show that every surjective local homeomorphism has a continuous section.

\begin{definition}
Let $X, Y$ be locally compact, Hausdorff, second countable, totally disconnected spaces. Let $\sigma: X \rightarrow Y$ be a surjective local homeomorphism. Define the set
\begin{align*}
R(\sigma) = \lbrace (u,v) \in X \times X: \sigma(u) = \sigma(v) \rbrace,
\end{align*}
and equip it with the operations $(u,v)(v,w) = (u,w)$ and $(u,v)^{-1} = (v,u)$. Then $R(\sigma)$ becomes a groupoid. We endow this set with the subspace topology from $X \times X$. Then $R(\sigma)$ is locally compact, Hausdorff, second countable, \'etale, and totally disconnected.
\end{definition}

\begin{definition}
We say that a locally compact, Hausdorff, second countable, totally disconnected groupoid is \emph{elementary} if it is isomorphic to the groupoid $R(\sigma)$ for some surjective local homeomorphism $\sigma: X \rightarrow Y$ between locally compact, Hausdorff, second countable, totally disconnected spaces. A groupoid is \emph{AF} if it is the increasing union of open elementary groupoids with the same unit space.
\end{definition}

Note that every AF groupoids is locally compact, Hausdorff, second countable, \'etale, totally disconnected, and amenable. If $G$ an AF groupoid is the increasing union of the open elementary subgroupoids $G_n$. Then $G$ is the inductive limit $G = \displaystyle\lim_{\longrightarrow} G_n$ \cite[Section III.1]{Renault}.

\begin{remark}
Given an elementary groupoid $G \cong R(\sigma)$ with $\sigma: X \rightarrow Y$, we apply the isomorphism to find a local homeomorphism $\widetilde{\sigma}: G^{(0)} \rightarrow Y$ such that $G = R(\widetilde{\sigma})$. So, we can suppose that $G = R(\sigma)$ without loss of generality.
\end{remark}

\begin{example}
\label{example:cinv0}
Let $X$ be a locally compact, Hausdorff, second countable, totally disconnected space, and let $\sigma: X \rightarrow X$ be a surjective local homeomorphism. Define the set
\begin{align*}
c^{-1}(0) = \lbrace (x, y) \in X^2: \sigma^n(x) = \sigma^n(y) \text{ for some } n \in \mathbb{N} \rbrace,
\end{align*}
and equip this set with the operations $(x,y)(y,z) = (x,z), (x,y)^{-1} = (y,x)$ and with the range a source maps $r,s: c^{-1}(0) \rightarrow X$ given by $r(x,y) = x$ and $s(x,y) = y$. We also equip the set with the subspace topology from $X \times X$. Then $c^{-1}(0)$ is an AF groupoids. In fact, it is the increasing union of the open elementary subgroupoids $R(\sigma^n)$.
\end{example}

\subsection{Homology in algebra} First we present the definition of homology groups from abstract algebra. A possible reference is  \cite[Section 17.1]{dummitfoote}.

\begin{definition}
Let $\mathcal{C}$ be a sequence of abelian group homomorphisms
$$
\begin{tikzcd}
0 &
\arrow{l}{}[swap]{d_0} C_0 &
\arrow{l}{}[swap]{d_1} C_1 &
\arrow{l}{}[swap]{d_2} \cdots
\end{tikzcd}
$$
such that $d_n \circ d_{n+1} = 0$ for all $n \geq 0$. The sequence is called a \emph{chain complex}. If $\mathcal{C}$ is a chain complex, its $n^{\text{th}}$ \emph{homology group} is the quotient group $\ker d_{n}/\mathrm{im\hphantom{.}}d_{n+1}$, and is denoted by $H_n(\mathcal{C})$.
\end{definition}

\begin{definition}
\label{def:homomorphism_complexes}
Let $\mathcal{A} = \lbrace A_n \rbrace_{n \geq 0}$ and $\mathcal{B} = \lbrace B_n \rbrace_{n \geq 0}$ be chain complexes. A \emph{homomorphism of complexes} $\alpha: \mathcal{A} \rightarrow \mathcal{B}$ is a set of homomorphisms $\alpha_n: A_n \rightarrow B_n$ such that for every $n$ the following diagram commutes:
$$
\begin{tikzcd}
\cdots &
A_n \arrow{l}{}[swap]{d_n} \arrow{d}{}[swap]{\alpha_n} &
A_{n+1} \arrow{l}{}[swap]{d_{n+1}} \arrow{d}{}[swap]{\alpha_{n+1}} &
\cdots \arrow{l}{}[swap]{d_{n+2}}\\
\cdots &
B_n \arrow{l}{}[swap]{d_n} &
B_{n+1} \arrow{l}{}[swap]{d_{n+1}} &
\cdots \arrow{l}{}[swap]{d_{n+2}}
\end{tikzcd}
$$
\end{definition}

\begin{proposition}
\label{prop:chain_homology_homomorphism}
A homomorphism $\alpha: \mathcal{A} \rightarrow \mathcal{B}$ of chain complexes induces a group homomorphism $H_n(\alpha): H_n(\mathcal{A}) \rightarrow H_n(\mathcal{B})$, for $n \geq 0$, given by
\begin{align}
\label{eqn:chain_homology_homomorphism}
H_n(\alpha)([a]) = [\alpha_n(a)],
\text{\hspace{10pt}for $a \in A_n$ with $d_n(a) = 0$.}
\end{align}
\end{proposition}

\subsection{Homology for groupoids} We define homology groups for Hausdorff, second countable, ample groupoids, as in \cite{FKPS, Matui}. We fix $X, Y$ to be locally compact, Hausdorff, second countable, totally disconnected spaces. 

Here, $C_c(X, \mathbb{Z})$ denotes the set of continuous and compactly supported functions on $X$ assuming integer values. Given $f \in C_c(X, \mathbb{Z})$ and a local homeomorphism $\sigma: X \rightarrow Y$, define for every $y \in Y$,
\begin{align*}
    \sigma_\ast(f)(y) = \sum_{x: \sigma(x) = y} f(x).
\end{align*}
This equation defines a function $\sigma_\ast(f) \in C_c(X, \mathbb{Z})$. For $n \geq 1$, we define the \emph{space of composable $n$-tuples} by
\begin{align*}
    G^{(n)}
    = \lbrace (g_1, \dots, g_n) \in G^n :
    s(g_i) = r(g_{i+1}) \text{ for } 1 \leq i < n \rbrace.
\end{align*}
We let $d_0^{(1)} = s$, $d_1^{(1)} = r$, and for $n \geq 2$ we define
\begin{align*}
d_i^{(n)}(g_1, \hdots, g_n) =
\begin{cases}
(g_2, \hdots, g_n) & i = 0,\\
(g_1, \hdots, g_{i-1}, g_i g_{i+1}, g_{i+2}, \hdots, g_n) & 1 \leq i \leq n - 1\\
(g_1, \hdots, g_{n-1}) & i = n.
\end{cases}
\end{align*}
Since $G$ is \'etale, the maps $d_i^{(n)}$ are local homeomorphisms. We use these maps to construct a chain complex. For $n \geq 1$, we define $\partial_n: C_c(G^{(n)}, \mathbb{Z}) \rightarrow C_c(G^{(n-1)}, \mathbb{Z})$ by
\begin{align*}
\partial_n = \sum_{i=0}^n (-1)^i (d_i^{(n)})_\ast.
\end{align*}
We set $\partial_0: C_c(G^{(0)}, \mathbb{Z}) \rightarrow 0$. The sequence
$$
\begin{tikzcd}
0 &
\arrow{l}{}[swap]{\partial_0} C_c(G^{(0)}, \mathbb{Z}) &
\arrow{l}{}[swap]{\partial_1} C_c(G, \mathbb{Z}) &
\arrow{l}{}[swap]{\partial_2} C_c(G^{(2)}, \mathbb{Z}) &
\arrow{l}{}[swap]{\partial_3}\cdots
\end{tikzcd}
$$
(also denoted $(C_c(G^{(\ast)}, \mathbb{Z}), \partial_\ast)$)  is a chain complex. This is true because, for $n \geq 1$,
\begin{align*}
\partial_n \partial_{n+1} = \sum_{i=0}^n \sum_{j=0}^{n+1}(-1)^{i+j} (d_i^{(n)} d_j^{(n+1)})_\ast = 0,
\end{align*}
which can be proved by showing that $d_i^{(n)} d_j^{(n+1)} = d_{j-1}^{(n)} d_i^{(n+1)}$ for all $0 \leq i < j \leq n + 1$, dividing the prove into different cases.

\begin{definition}
\label{def:gpdhomology}
For $n \geq 0$, let $H_n(G)$ be the $n^{\text{th}}$ homology group of the chain complex $(C_c(G^{(\ast)}, \mathbb{Z}), \partial_\ast)$. In addition, we define the set of positive elements
\begin{align*}
H_0(G)^+ = \lbrace [f] \in H_0(G) : f(x) \geq 0 \text{ for all }x \in G^{(0)} \rbrace.
\end{align*}
The operations on $H_0(G)$ are given by $[f_1] + [f_2] = [f_1 + f_2]$ and $-[f] = [-f]$.
\end{definition}

 The group $H_0(G)$ is not necessarily an ordered group because the intesection $H_0(G)^+ \cap (- H_0(G)^+)$ may be different from $\lbrace 0 \rbrace$. See \cite[Remark 3.2]{Matui}.

 Given a homomorphism $\rho: G \rightarrow H$ between \'etale groupoids, we define $\rho^{(0)} : G^{(0)} \rightarrow H^{(0)}$ by $\rho^{(0)} = \rho \vert_{G^{(0)}}$. For $n \geq 1$, we define the map $\rho^{(n)}: G^{(n)} \rightarrow H^{(n)}$ by
 \begin{align*}
\rho^{(n)}(g_1, \hdots, g_n) = (\rho(g_1), \hdots, \rho(g_n)),
\text{\hspace{10pt}for $(g_1, \hdots, g_n) \in G^{(n)}$.}
\end{align*}

We prove the lemma below, stated in \cite[Subsection 3.2]{Matui} for \'etale groupoids.

\begin{lemma}
\label{lemma:rhon-equivalent}
 
Let $G, H$ be locally compact étale groupoids and let $\rho: G \rightarrow H$ be a homomorphism. Then the following are equivalent:
\begin{enumerate}[(i)]
\item $\rho^{(0)}$ is a local homeomorphism,
\item $\rho$ is a local homeomorphism,
\item $\rho^{(n)}$ is a local homeomorphism for all $n \in \mathbb{N}$.
\end{enumerate}
\end{lemma}
\begin{proof}
    The implication $\mathrm{(iii)} \Rightarrow \mathrm{(ii)} \Rightarrow \mathrm{(i)}$ is straightforward. To prove $\mathrm{(i)} \Rightarrow \mathrm{(ii)}$, first we show that $\rho$ is an open map. Let $\mathcal{U}$ be an open bisection such that the restriction $\rho\vert_{r(\mathcal{U})}^{(0)}: r(\mathcal{U}) \rightarrow \rho(r(\mathcal{U}))$ is a homeomorphism. Let $V \subset \overline{V} \subset \mathcal{U}$ be an open set, and let $g_n$ be a sequence in $\mathcal{U} \setminus V$ converging to an $g \in \mathcal{U} \setminus V$. Since $\mathcal{U}$ is \'etale, we have $r(g) \notin r(V)$. Since $\rho$ is injective in $\mathcal{U}$, we have $r \circ \rho(g) \notin \rho \circ r(\mathcal{V})$, which implies that $\rho(g) \notin \rho(V)$. Then $\rho(V)$ is open. It is straightforward that $\rho$ is injective on open bisections. Let $\mathcal{U}_1 \subset G$, $V_1 \subset H$ be open bisections such that $\rho(\mathcal{U}_1) \subset V_1$. One can show that, locally, $\rho$ can be identified with the map $r\vert_{V_1}^{-1} \circ \rho \circ r\vert_{\mathcal{U}_1}$.

    To show $\mathrm{(ii)} \Rightarrow \mathrm{(iii)}$ for $n \geq 2$, choose open bisections $\mathcal{U}_1, \dots, \mathcal{U}_n$ such that $r(\mathcal{U}_i) = s(\mathcal{U}_{i+1})$ for $i = 1, \dots, n-1$, and show that $\rho^n(\mathcal{U} \cap G^n) = \rho^n(\mathcal{U}) \cap H^{(n)}$, where
    \begin{align*}
        \rho^{n} = \underbrace{\rho \times \dots \times \rho}_{\text{$n$ times}}.
    \end{align*}
    Since $\rho^{(n)}$ is continuous and injective on $\mathcal{U} \cap G^{(n)}$, then $\rho$ is a local homeomorphism.
\end{proof}

\begin{proposition}
    Let $\rho : G \rightarrow H$ be a homomorphism that is also a local homeomorphism. Then, for each $n \geq 0$, $\rho$ induces a homomorphism $H_n(\rho) : H_n(G) \rightarrow H_n(H)$ defined by
    \begin{align*}
        H_n(\rho)([f]) = [\rho_\ast^{(n)}(f)]
        \hspace{7pt} \text{for } f \in C_c(G^{(n)}, \mathbb{Z}).
    \end{align*}
\end{proposition}

\begin{remark}
\label{remark:inducedhomomorphismH0}
In this paper, for a function $\tau: C_c(G^{(0)}, \mathbb{Z}) \rightarrow C_c(H^{(0)}, \mathbb{Z})$ such that $f_1 \sim f_2$ in $H_0(G)$ imples $\tau(f_1) \sim \tau(f_2)$ in $H_0(H)$, we use the notation $[\tau]$ to denote the function $[\tau]: H_0(G) \rightarrow H_0(H)$ given by $[\tau]([f]) = [\tau(f)]$. Moreover, for a homomorphism $\rho: G \rightarrow H$ that is also a local homeomorphism, we have $H_0(\rho) = [\rho_\ast^{(0)} ]$.    
\end{remark}

An important tool in the study of homology groups is the notion of homological similarity, that can be applied to show that two groupoids have isomorphic homology groups.

\begin{definition}
\begin{enumerate}[(i)]
\item Two homomorphisms $\rho, \sigma$ from $G$ to $H$ are said to be \emph{similar} if there exists a continuous map $\theta: G^{(0)} \rightarrow H$ such that
\begin{align}
\label{eqn:similar}
\theta(r(g)) \rho(g) = \sigma(g)\theta(s(g)),
\end{align}
for all $g \in G$.
\item $G$ and $H$ are said to be \emph{homologically similar} if there exist homomorphisms $\rho: G \rightarrow H$ and $\sigma: H \rightarrow G$ that are local homeomorphisms and such that $\sigma \circ \rho$ is similar to $\mathrm{id}_G$ and $\rho \circ \sigma$ is similar to $\mathrm{id}_H$.
\end{enumerate}
\end{definition}

\begin{proposition}
\label{prop:similarhomomorphisms}
\cite[Proposition 3.5]{Matui}
Let $n \geq 1$. If $\rho, \sigma: G \rightarrow H$ are similar homomorphisms and local homeomorphisms, then $H_n(\rho) = H_n(\sigma)$.
\end{proposition}

\begin{corollary}
\label{corollary:isoHn}
\cite[Proposition 3.5]{Matui}
If $G$ and $H$ are homologically similar, then they have isomorphic homology groups. In particular, if $n \geq 1$, $\rho: G \rightarrow H$ and $\sigma: H \rightarrow G$ are local homeomorphisms and homomorphisms satisfying \eqref{eqn:similar}, then $H_n(\rho): H_n(G) \rightarrow H_n(H)$ is an isomorphism with inverse $H_n(\sigma)$. Moreover, both $H_0(\rho)$ and $H_0(\sigma)$ preserve positive elements.
\end{corollary}

%\section{Homology Groups and K-theory}
\section{Homology for AF groupoids}
\label{section:homology}

In this section, we study the properties of the homology groups of AF groupoids. Like K-theory, homology groups are continuous with respect to inductive limits. We start this subsection by looking at the building blocks of AF groupoids, i.e., groupoids of the form $R(\sigma)$ and finding a property that helps us find the equivalence in $H_0(R(\sigma))$.

\begin{lemma}
\label{lemma:traceH0Rsigma}
Let $\sigma: X \rightarrow Y$ be a surjective local homeomorphism. Then two functions $f_1, f_2 \in C_c(X, \mathbb{N})$ are equivalent in $H_0(R(\sigma))$ if, and only if,
\begin{align}
\label{eqn:equivconditionH0Rsigma}
\sum_{u: \sigma(u) = \sigma(x)} f_1(u)
= \sum_{u: \sigma(u) = \sigma(x)} f_2(u)
\hspace{15pt}\text{for all $x \in X$.}
\end{align}
\end{lemma}
\begin{proof}
First, let us define a map $T: C_c(X, \mathbb{N}) \rightarrow C_c(X, \mathbb{Z})$ which is going to be used in the proof. Fix a continuous section $\varphi: Y \rightarrow X$ of $\sigma$. Let $T: C_c(X, \mathbb{N}) \rightarrow C_c(X, \mathbb{N})$ be given by
\begin{align}
\label{eqn:Tfx}
T(f)(x) = 1_{X_\varphi}(x) \sum_{u: \sigma(u) = \sigma(x)} f(u),
\end{align}
where $X_\varphi = \varphi(Y)$. To show that $T$ is well-defined, Fist we consider functions of the form $1_V$, where $V \subset X$ be a compact open subset on which $\sigma$ is injective. We claim that $T(1_V) = 1_{\varphi(\sigma(V))}$. Let $x \in X$. Suppose that $T(1_V)(x) \neq 0$. Equation \eqref{eqn:Tfx} implies that $x \in X_\varphi$ and that there exists an $u \in X_\varphi$ such that $u \in V$ and $\sigma(u) = \sigma(x)$. Since $\sigma$ is injective on $V$, such $u$ is unique, which implies that $x = \varphi(\sigma(u))$. Thus $1_{\varphi(\sigma(V))}(x) = 1_V(x)$.

Now suppose that $T(1_V)(x) = 0$. Then $x \notin X_\varphi$ or $\sum_{u: \sigma(u) = \sigma(x)} 1_V(u) = 0$. If $x \notin X_\varphi$, then $x \notin\varphi(\sigma(V))$ because $\varphi(\sigma(V)) \subset X_\varphi$. Then $1_{\varphi(\sigma(V))}(x) = 0$. If $\sum_{u: \sigma(u) = \sigma(x)} 1_V(u) = 0$, then $\varphi(\sigma(x)) \notin V$, which implies that $T(1_V)(x) = 1_{\varphi(\sigma(V))}(x) = 0$. Therefore $T(1_V) = 1_{\varphi(\sigma(V))}$.

Note that every function $f \in C_c(X, \mathbb{N})$ is spanned by  function of the form $1_V$, as considered above. Since $T$ is a semigroup homomorphism, then $T(f) \in C_c(X, \mathbb{N})$. Therefore, $T$ is well-defined.

Let $f_1, f_2 \in C_c(X, \mathbb{N})$ be equivalent in $H_0(R(\sigma))$. Then there exists $F \in C_c(R(\sigma), \mathbb{Z})$ such that, for all $u \in X$,
\begin{align*}
f_1(u)
&= f_2(u) + \partial_1 F(u) \\
&= f_2(u) + s_\ast(F)(u) - r_\ast(F)(u) \\
&= f_2(u) + \sum_{v: \sigma(v) = \sigma(u)} [F(v,u) - F(u,v)].
\end{align*}
Fix $x \in X$. Then
\begin{align*}
\sum_{u: \sigma(u) = \sigma(x)} f_1(u)
&= \sum_{u: \sigma(u) = \sigma(x)} f_2(u) + \sum_{\substack{u,v \\ \sigma(u) =  \sigma(x) \\ \sigma(v) = \sigma(x)}} (F(v,u) - F(u,v)) \\
&= \sum_{u: \sigma(u) = \sigma(x)} f_2(u) + \sum_{\substack{u,v \\ \sigma(u) =  \sigma(x) \\ \sigma(v) = \sigma(x)}} F(v,u) - \sum_{\substack{u,v \\ \sigma(u) =  \sigma(x) \\ \sigma(v) = \sigma(x)}} (F(u,v).
\end{align*}
Note that the second and the third sums of the right-hand side are equal. Then \eqref{eqn:equivconditionH0Rsigma} holds.

Conversely, suppose that \eqref{eqn:equivconditionH0Rsigma} holds. Define the set
\begin{align*}
K = \lbrace x \in X : x \in \mathrm{supp\hphantom{.}}f_1 \cup \mathrm{supp\hphantom{.}} f_2 \text{\hspace{7pt} and \hspace{7pt}} T(f_1)(\varphi(\sigma(x))) \neq 0 \rbrace.
\end{align*}
This set is compact because it is the intesection of the compact set $\mathrm{supp\hphantom{.}}f_1 \cup \mathrm{supp\hphantom{.}} f_2$ and the closed set $(T(f_1) \circ \varphi \circ \sigma)^{-1}(\mathbb{Z} \setminus \lbrace 0 \rbrace)$.

Let $F \in C_c(R(\sigma), \mathbb{Z})$ be defined by
\begin{align*}
F(x,y)
&= \begin{cases}
\displaystyle\frac{f_1(y)f_2(x)}{T(f_1)(\varphi(\sigma(x)))}
& \text{if } (x,y) \in K \times \mathrm{supp\hphantom{.}}f_2,\\
0 &\text{otherwise.}
\end{cases}
\end{align*}
Then, for $x \in K$, we have
\begin{align}
\sum_{y: \sigma(y) = \sigma(x)} F(x,y)
&= \frac{f_2(x)}{T(f_1)(\varphi(\sigma(x)))} \sum_{y: \sigma(y) = \sigma(x)} f_1(y) \nonumber\\
&= \frac{f_2(x)}{T(f_1)(\varphi(\sigma(x)))} T(f_1)(\varphi(\sigma(x)))  \nonumber\\
&= f_2(x). \label{eqn:partialFf1}
\end{align}
If $x \in X \setminus K$, we have
\begin{align}
\label{eqn:partialFf1zero}
\sum_{y: \sigma(y) = \sigma(x)} F(x,y) = 0 = f_2(x).
\end{align}
By similar arguments, we have
\begin{align}
\label{eqn:partialFf2}
\sum_{y: \sigma(y) = \sigma(x)} {F(y,x)} = f_1(x)
\hspace{20pt}\text{for $x \in X$.}
\end{align}
The equations \eqref{eqn:partialFf1}, \eqref{eqn:partialFf1zero} and \eqref{eqn:partialFf2} imply that, for all $x \in X$, $f_1(x) = f_2(x) + \partial_1 F(x).$ Therefore $f_1 \sim f_2$. 
\end{proof}

The proof of the following lemma is straightforward, and this lemma will be applied in the proof of Theorem \ref{thm:main}.
\begin{lemma}
\label{lemma:partial1F}
Let $G = \displaystyle\lim_\rightarrow G_n$. For each $n$, let $\partial_1^{(n)}: C_c(G_{n}, \mathbb{Z}) \rightarrow C_c(G^{(0)}, \mathbb{Z})$ be defined by $\partial_1^{(n)} = (s\vert_{G_n})_\ast - (r\vert_{G_n})_\ast$. Analogously, let $\partial_1: C_c(G, \mathbb{Z}) \rightarrow C_c(G^{(0)}, \mathbb{Z})$ be given by $\partial_1 = s_\ast - r_\ast$. Let $F \in C_c(G_n, \mathbb{Z})$. Then for all $k \geq 0$, we have $\partial_1 F = \partial_1^{(n)} F = \partial_1^{(n+k)} F.$
\end{lemma}

Now we show that $H_0$ is continuous with respect to inductive limits.

\begin{lemma}
\label{lemma:limH0}
Let $G = \displaystyle\lim_\rightarrow G_n$. Then $H_0(G) = \displaystyle\lim_\rightarrow H_0(G_n)$ as an inductive limit of groups. Moreover, if each $H_0(G_n)$ is an ordered group, then $H_0(G)$ is an ordered group.
\end{lemma}
\begin{proof}
    Let $[\hphantom{h}]_n$ and $[\hphantom{h}]$ be the equivalence classes in $H_0(G_n)$ and $H_0(G)$, respectively. Let $i_n: H_0(G_n) \rightarrow H_0(G_{n+1})$ and $\alpha_n: H_0(G_n) \rightarrow H_0(G)$ be the inclusion maps. It is straightforward to show that these maps are well-defined and satisfy $\alpha_{n+1} \circ i_n = \alpha_n$. Let $B$ be an group such that for all $n$, there is a homomorphism $\lambda_n: H_0(G_n) \rightarrow B$ satisfying
    \begin{align}
        \label{eqn:lambdanB}
        \lambda_n = \lambda_{n+1} \circ i_n.
    \end{align}
    If the groups $H_0(G_n)$. are ordered, suppose that $B$ is ordered and that the maps $\lambda_n$ preserves positive elements. Let $\lambda: H_0(G) \rightarrow B$ be defined by $\lambda([f]) = \lambda_0(f)$. The equations \eqref{eqn:lambdanB} imply that $\lambda$ is well-defined. Indeed, let $f_1, f_2 \in C_c(G^{(0)}, \mathbb{Z})$ be such that $f_1, f_2$ are equivalent in $H_0(G_n)$. Then there exists $F \in C_c(G_n, \mathbb{Z})$ such that $f_1 = f_2 + \partial_1 F$. Then there is an $n$ such that $F \in C_c(G_n, \mathbb{Z})$. Then, by Lemma \ref{lemma:partial1F}, $f_1 = f_2 + \partial_1^{(n)} F$. Applying \eqref{eqn:lambdanB} multiply times, we have
    \begin{align*}
        \lambda_0(f_1)
        = \lambda_1 \circ i_0(f_1) = \lambda_1([f_1]_1) = \lambda_2([f_1]_2) = \dots
        = \lambda_n([f_1]_n)
        = \lambda_n([f_2]_n)
        = \dots 
        = \lambda_0(f_2).
    \end{align*}
    Thus $\lambda$ is well-defined. It is straightforward to show that $\lambda$ is the only homomorphism such that $\lambda \circ \alpha_n = \lambda_n$. Therefore $H_0(G) = \displaystyle\lim_\rightarrow H_0(G_n)$. If each $H_0(G_n)$ is an ordered group, then $H_0(G)$ is an ordered group.
\end{proof}

Now we show that, for AF groupoids, the zeroth homology group is an ordered group.
\begin{proposition}
Let $G$ be an AF groupoid. Then $H_0(G)$ is an ordered group.
\end{proposition}
\begin{proof}
We begin by showing that, for a surjective local homeomorphism $\sigma: X \rightarrow Y$ of locally compact, Hausdorff, second countable, totally disconnected spaces, the group $H_0(R(\sigma))$ is an ordered group. Indeed, it is straightforward from the definition that $H_0(R(\sigma)) = H_0(R(\sigma))^+ - H_0(R(\sigma))^+$. We now prove the intersection
\begin{align*}
H_0(R(\sigma))^+ \cap (-H_0(R(\sigma))^+) = \lbrace 0 \rbrace.
\end{align*}

Choose a function $f \in C_c(X, \mathbb{Z})$ with $[f] \in H_0(R(\sigma))^+ \cap (-H_0(R(\sigma))^+)$. Then there are two functions $f_1, f_2 \in C_c(X, \mathbb{Z})$ assuming non-negative values, such that $[f] = [f_1] = [-f_2].$ Choose $x \in X$. By Lemma \ref{lemma:traceH0Rsigma},
\begin{align*}
0
\leq f_1(x)
\leq \sum_{u: \sigma(u) = \sigma(x)} f_1(u)
= -  \sum_{u: \sigma(u) = \sigma(x)} f_2(u)
\leq - f_2(x)
\leq 0.
\end{align*}
Then $f_1(x) = f_2(x) = 0$. Since $x$ is arbitrary, we have $f_1 = f_2 = 0$. This implies that $H_0(R(\sigma))^+ \cap (-H_0(R(\sigma))^+) = \lbrace 0 \rbrace$. Therefore, $H_0(R(\sigma))$ is an ordered group.

Now let $G$ be an AF groupoid. Then there exists a sequence of surjective local homeomorphisms $\sigma_n: G^{(0)} \rightarrow Y_n$, where $G^{(0)}$ and $Y_n$ are locally compact, Hausdorff, second countable and totally disconnected, such that $G = \displaystyle\lim_\rightarrow R(\sigma_n)$. Since each $H_0(R(\sigma_n))$ is an ordered group, then Lemma \ref{lemma:limH0} implies that $H_0(G)$ is an ordered group.
\end{proof}

%\section{K-theory for AF Groupoids}
\section{Main result}
\label{section:KtheoryAFgroupoids}

Here we extend the isomorphism \cite[Corollary 5.2]{FKPS} between the K-theory and the homology group for AF groupoids, showing that the map preserves positive elements, and giving a formula for the isomorphism.

Although it is often difficult to find the K-theory of an arbitrary C*-algebra, finding the K-theory of $M_n$ is a easy. In this case, we have the isomorphism $K_0(M_n) \rightarrow \mathbb{Z}$ given by $[p]_0 \mapsto \mathrm{tr}\hphantom{.}p$, where $\mathrm{tr}\hphantom{.}p$ is the matrix trace of $p \in \mathcal{P}_\infty(M_n)$. Similarly, for a locally compact, Hausdorff, second countable, totally disconnected space $X$, the K-theory of $C_0(X, M_n)$ is also described in terms of a trace operator.

To find the K-theory of the C*-algebra of an AF groupoid, we first consider the building blocks of the groupoid. In other words, we consider an elementary groupoid of the form $R(\sigma)$. We show that $C^*(R(\sigma))$ is locally approximated by groupoids of the form $C_0(V, M_n)$. Then we apply approximation to find a formula for the $K_0$-group of $C^*(R(\sigma))$ and an isomorphism from this group to the corresponding homology group. 

Throughout this section, we write $[f_{jk}]$ instead of $[f_{jk}]_{j,k = 1, \dots, n}$ to simplify our notation. We also assume that the topological spaces $X, Y, X_n$ are locally compact, Hausdorff, second countable, and totally disconnected. The map $\sigma: X \rightarrow Y$ is a surjective local homeomorphism.

\begin{remark}
Unless otherwise stated, we assume that $X$ has a compact open subset $X_0$ such that $\sigma^{-1}(\sigma(X_0)) = X_0$ and $\sigma$ is injective on $X \setminus X_0$. Later on, we prove that an arbitrary elementary groupoid is the inductive limit of groupoids of the form $R(\sigma)$ where this condition is satisfied.
\end{remark}

\begin{lemma}
    \label{lemma:covmap-sets}
    Let $\varphi: Y \rightarrow X$ be a continuous section of $\sigma$, and let $y \in Y$. Then $y$ has a clopen neighbourhood $W$, there is a positive integer $n$, and there are disjoint subsets $\mathcal{U}_1, \dots, \mathcal{U}_n \subset X$ such that
    \begin{enumerate}[(i)]
        \item $\sigma^{-1}(W) = \bigcup_{i=1}^n \mathcal{U}_i$,
        \item $\sigma\vert_{\mathcal{U}_i}: \mathcal{U}_i \rightarrow W$ is a homeomorphism for all $i$, and
        \item $\sigma^{-1}(W) \cap X_\varphi = \mathcal{U}_1$.
    \end{enumerate}
    If $y \in \sigma(X_0)$, then $W$ and the sets $\mathcal{U}_i$ are compact. If $y \in Y \setminus \sigma(X_0)$, we let $n = 1$, $W = Y \setminus \sigma(X_0)$ and $\mathcal{U}_1 = X \setminus X_0$.
\end{lemma}
\begin{proof}
    The proof of the case $y \in Y \setminus \sigma(X_0)$ follows directly from the fact that $\sigma$ is injective on $X \setminus X_0$. Suppose $y \in \sigma(X_0)$. Since $\sigma$ is a local homeomorphism, $X_0$ is compact and $\sigma^{-1}(\sigma(X_0)) = X_0$, then $y$ has a compact open neighbourhood $W$ and there are finite disjoint compact open subsets $\mathcal{O}_1, \dots, \mathcal{O}_n$ with
    \begin{align*}
        \sigma^{-1}(W) = \bigcup_{i=1}^n \mathcal{O}_i,
    \end{align*}
    and such that $\sigma\vert_{\mathcal{O}_i}: \mathcal{O}_i \rightarrow W$ is a homeormorphism for all $i$.

    Define the set $\mathcal{U}_1 = \sigma^{-1}(W) \cap X_\varphi$. Then $\sigma$ is injective on $\mathcal{U}_1$. Moreover, $\sigma(\mathcal{U}_1) = W$ because $\sigma(\sigma^{-1}(X) \cap X_\varphi)) = W \cap Y = W$. Therefore, $\sigma\vert_{\mathcal{U}_1} : \mathcal{U}_1 \rightarrow W$ is a homeomorphism. By rearranging the sets $\mathcal{O}_1, \dots, \mathcal{O}_n$ if necessary, let $m \leq n$ be such that $\mathcal{U}_1 \cap \mathcal{O}_i = \emptyset \Leftrightarrow j > m$. Then we have the disjoint union
    \begin{align*}
        \mathcal{U}_1 = \bigcup_{i=1}^m \mathcal{U}_i \cap \mathcal{O}_i.
    \end{align*}
    For all $i, j = 1, \dots, m$, let $\Omega_{ij} = \sigma^{-1}(\sigma(\mathcal{U}_1 \cap \mathcal{O}_j)) \cap \mathcal{O}_i$. Then these sets are disjoint. For all $i$, we have
    \begin{align}
    \label{eqn:sigmaW1}
        \bigcup_{i=1}^m \Omega_{ii}
        = \bigcup_{i=1}^m \sigma^{-1}(\sigma(\mathcal{U}_1 \cap \mathcal{O}_i)) \cap \mathcal{O}_i
        = \bigcup_{i=1}^m \mathcal{U}_1 \cap \mathcal{O}_i
        = \mathcal{U}_1,
    \end{align}
    because $\sigma$ is injective on each $\mathcal{O}_i$. Similarly,
    \begin{align}
    \label{eqn:sigmaW2}
        \bigcup_{j=1}^m \Omega_{ij}
        = \sigma^{-1} \left( \sigma\left( \bigcup_{j=1}^m \mathcal{U}_1 \cap \mathcal{O}_j \right) \right) \cap \mathcal{O}_i
        = \sigma^{-1}(\sigma(\mathcal{U}_1)) \cap \mathcal{O}_i
        = \sigma^{-1}(W) \cap \mathcal{O}_i
        = \mathcal{O}_i.
    \end{align}
    Also,
    \begin{align}
    \label{eqn:sigmaW3}
        \sigma(\Omega_{ij}) = \sigma \left( \sigma^{-1} \left( \sigma( \mathcal{U}_1 \cap \mathcal{O}_j)\right) \cap \mathcal{O}_i \right)
        = \sigma(\mathcal{U}_1 \cap \mathcal{O}_j)
        = \sigma(\Omega_{1j}).
    \end{align}
    Since $\sigma$ is injective on $X_\varphi$, \eqref{eqn:sigmaW1} implies that $\Omega_{ij} \cap X_\varphi = \emptyset$ for $i \geq j$.
    
    For $i = 2, \dots, m$, let $\mathcal{U}_i = \bigcup_{j=1}^m \Omega_{ij}$. For $i = m + 1, \dots, n$, let $\mathcal{U}_i = \mathcal{O}_i$. Then equations \eqref{eqn:sigmaW1}, \eqref{eqn:sigmaW2}, \eqref{eqn:sigmaW3} imply items (i)--(iii).
\end{proof}

\begin{lemma}
\label{lemma:traceCXMn}
Let $p, q \in C_0(X, M_n)$ be two projections. Then $p, q \in C_c(X, \mathbb{N})$ and $\mathrm{tr}\hphantom{.}p, \mathrm{tr}\hphantom{.}q \in C_c(X, \mathbb{N})$. Moreover, the following are equivalent:
\begin{enumerate}[(i)]
\item $p \sim q$,
\item $\mathrm{tr}\hphantom{.} p = \mathrm{tr}\hphantom{.} q$,
\item For every compact subset $K \subset X$ such that $p, q$  vanish outside $K$, there is a $u \in C_c(X, M_n)$ supported on $K$ such that
\begin{align*}
u^*u = uu^* = \mathrm{id} 1_K
\hspace{15pt}\text{and}\hspace{15pt}
 p = u^* q u.
\end{align*}
\end{enumerate}
\end{lemma}
\begin{proof}
    Let $p \in C_0(X, M_n)$ be a projection.Let $K \subset X_\varphi$ be a compact subset such that $\Vert p(x) \Vert < 1$ for every $x \notin K$. It follows from \cite[Propositions 2.2.4 and 2.2.7]{RLL-Ktheory} that $\mathrm{tr}\hphantom{.}p(x) = 0$ for all $x \notin K$. Then $\mathrm{tr}\hphantom{.}p \in C_c(X_\varphi, \mathbb{N})$ and $p \in C_c(X, \mathbb{N})$.

    The implication $\mathrm{(i)} \Rightarrow \mathrm{(ii)}$ is straightforward. We show $\mathrm{(ii)} \Rightarrow \mathrm{(iii)}$. Suppose $p \sim q$. Choose a compact subset $K$ containing the supports of $p$ and $q$. By continuity, there is a collection of elements $x_1, \hdots, x_m \in K$ and a partition $V_1, \hdots, V_m$ of $K$ by disjoint compact open subsets such that, for each $i = 1, \hdots, m$, we have
    \begin{align}
        \label{eqn:traceCXMn1}
        x_i \in V_i,
        \hspace{10pt}
        \Vert p(y) - p(x_i) \Vert < \tfrac{1}{2},
        \text{\hspace{10pt}and }
        \Vert q(y) - q(x_i) \Vert < \tfrac{1}{2},
        \text{\hspace{10pt}for }
        y\in V_i.
    \end{align}

    Let $i = 1, \hdots, m$. By assumption, $\mathrm{tr}(p(x_i)) = \mathrm{tr}(q(x_i))$. Then there exists an unitary $v_i \in M_n$ such that $p(x_i) = v_i^* q(x_i) v_i$.

    Let
    \begin{align*}
        \widetilde{p} = \displaystyle\sum_{i=1}^m p(x_i)1_{V_i},
        \text{\hspace{10pt}}
        \widetilde{q} = \displaystyle\sum_{i=1}^m q(x_i)1_{V_i}
        \text{\hspace{10pt}and\hspace{10pt}}
        v = \sum_{i=1}^m v_i 1_{V_i}.
    \end{align*}
    Then $\widetilde{p}$ and $\widetilde{q}$ are projections in $C_0(X, M_n)$. Also, $\widetilde{p} = v^* \widetilde{q} v$ and $v^*v = vv^* = \mathrm{id}1_K$. Considering the unital C*-algebra $C(K, M_n) \subset C_0(X, M_n)$, it follows from \cite[Propositions 2.2.4 and 2.2.7]{RLL-Ktheory} that there exists $u_p, u_q \in C_0(X, M_n)$ with support in $K$ such that
    \begin{align*}
        p = u_p^* \widetilde{p} u_p
        \hspace{15pt}\text{and}\hspace{15pt}
        \widetilde{q} = u_{q}^* q u_{q}.
    \end{align*}
    Note that $u_p$ and $u_{q}$ are unitary in $A$, not necessarily in $C_0(X, M_n)$. Then
    \begin{align*}
        u_p^* u_p = u_p u_p^* = \mathrm{id}1_K
        \hspace{15pt}\text{and}\hspace{15pt}
        u_{q}^* u_{q} = u_{q} u_{q}^* = \mathrm{id} 1_K.
    \end{align*}
    Let $u = u_{q} v u_p$. Then $u^*u = uu^* = \mathrm{id} 1_K$, and
    \begin{align*}
        u^* q u
        = u_p^* v^* u_q^* q u_q v u_p 
        = u_p^* v^* \widetilde{q} v u_p 
        = u_p^* \widetilde{p} u_p 
        = p.
    \end{align*}
    Therefore (iii) holds. To show $\mathrm{(iii)} \Rightarrow \mathrm{(i)}$, let $v = qu$. Then $v^* v = p$ and $vv^* = q$.
\end{proof}

In the following proposition we show that the C*-algebra can be locally isomorphic to C*-algebras of the form $C_0(V, M_n)$, where we can use the trace operator to describe equivalence classes of projections. We also apply this isomorphism to induce a trace operator that describes the equivalence of projections in $\mathcal{P}_\infty(C^*(R(\sigma)))$.

\begin{proposition}
\label{prop:isoRsigmacompact}
Let $\varphi: Y \rightarrow X$ be a continuous section of $\sigma$. Then there is a positive integer $N$, there are disjoint open sets $V_1, \dots, V_N$ with $X_\varphi = V_1 \cup \dots \cup V_N$, there are positive integers $n_1, \dots, n_N$, and there is an isomorphism $\mu: C^*(R(\sigma)) \rightarrow \bigoplus_{i=1}^N C_0(V_i, M_{n_i})$ such that
\begin{enumerate}[(i)]
\item for $f \in C^*(R(\sigma))$ and $x \in X_\varphi$, we have
\begin{align*}
\mathrm{tr}\hphantom{.}\mu(f)(x)
= \sum_{u: \sigma(u) = \sigma(x)} f(u),
\end{align*}
where we apply Renualt's $j$ map \cite[Proposition II.4.2]{Renault} to identify $C^*(R(\sigma))$ as a subset of $C_0(R(\sigma))$.
\item For a positive integer $k$, for $F \in M_k(C^*(R(\sigma)))$, and for $x \in X_\varphi$, we have
\begin{align*}
\mathrm{tr}\hphantom{.}\mu_k(F)(x)
= \sum_{i=1}^k \mathrm{tr}\hphantom{.}\mu(F_{ii})(x),
\end{align*}
where $\mu_k: M_k(C^*(R(\sigma))) \rightarrow \bigoplus_{i=1}^N C_0(V_i, M_{kn_i})$ is the isomorphism given by
\begin{align*}
\mu_k(F)(x) = [\mu(F_{rs})(x)],
\hspace{15pt}
\text{for }
F = [F_{rs}] \in M_k(C^*(R(\sigma))), \text{ } x \in X_\varphi.
\end{align*}
\end{enumerate}
For the next items, we fix positive integers $k,l$, and choose projections $p \in \mathcal{P}_k(C^*(R(\sigma)))$, $q \in \mathcal{P}_l(C^*(R(\sigma)))$. Then
\begin{enumerate}[(i)]
\setcounter{enumi}{2}
\item $\mathrm{tr}\hphantom{.}\mu_k(p) \in C_c(X_\varphi, \mathbb{N})$.

\item $\mathrm{tr}\hphantom{.} \mu_{k+l}(p \oplus q) = \mathrm{tr}\hphantom{.} \mu_{k}(p) + \mathrm{tr}\hphantom{.} \mu_{l}(q)$.

\item $p \sim q
\Leftrightarrow
\mathrm{tr} \hphantom{.}\mu_k(p) = \mathrm{tr}\hphantom{.} \mu_{l}(q).$

\item $\mathrm{tr}\hphantom{.}\mu_k(p) = 0 \Rightarrow p = 0_k$.
\end{enumerate}
\end{proposition}
\begin{proof}
    Let $y \in Y$. By Lemma \ref{lemma:covmap-sets}, $y$ has a clopen neighbourhood $\widetilde{W}_y$, there is a positive integer $n_y$, and there are $n_y$ disjoint clopen subsets $\widetilde{\mathcal{U}}_1^{(y)}, \widetilde{\mathcal{U}}_2^{(y)}, \dots, \widetilde{\mathcal{U}}_{n_y}^{(y)} \subset X$ such that
    \begin{enumerate}[(a)]
        \item $\sigma^{-1}(\widetilde{W}_y) = \bigcup_{j=1}^{n_y} \widetilde{\mathcal{U}}_j^{(y)}$,

        \item $\sigma\vert_{\widetilde{\mathcal{U}}_j^{(y)}}: \widetilde{\mathcal{U}}_j^{(y)} \rightarrow \widetilde{W}_y$ is a homeomorphism for all $j$, and

        \item $\sigma^{-1}(\widetilde{W}_y) \cap X_\varphi = \widetilde{\mathcal{U}}_1^{(y)}.$
    \end{enumerate}

    Since $\sigma(X_0)$ is compact and Hausdorff, we choose $N - 1$ elements $y_1, \dots, y_{N-1} \in \sigma(X_0)$ such that $\sigma(X_0) = \widetilde{W}_{y_1} \cup \dots \cup \widetilde{W}_{y_{N-1}}$ and such that for all $i,j = 1, \dots, N-1$,
    \begin{align}
        \label{eqn:yiWj}
        y_i \in \widetilde{W}_{y_j} \Leftrightarrow i = j.
    \end{align}
    Let $W_1 = \widetilde{W}_1$ and, for $i= 2, \dots, N-1$, define $W_i = \widetilde{W}_i \setminus \bigcup_{j=1}^{i-1} W_j$. It follows from property \eqref{eqn:yiWj} that the sets $W_i$ are non-empty. Moreover, $\sigma(X_0)$ is the disjoint union $\sigma(X_0) = W_1 \cup \dots \cup W_{N-1}$.

    If $\sigma(X_0) = Y$, we make an abuse of notation and replace $N-1$ by $N$. Otherwise, let $y_N \in Y \setminus \sigma(X_0)$, and set $n_{N} = n_{y_N} = 1$, $W_{N} = Y \setminus \sigma(X_0)$, and $\mathcal{U}_1^{(y_N)} = X \setminus X_0$. Then properties (a), (b) and (c) also hold by replacing $\widetilde{W}_y$, $n_y$ and $\widetilde{\mathcal{U}}_j^{(y)}$ by $W_N$, $n_N$ and $\mathcal{U}_1^{(y_n)}$, respectively. In either case, we have the disjoint union
    \begin{align*}
        Y = \sigma(X_0) \cup (Y \setminus \sigma(X_0)) = W_1 \cup \dots \cup W_N.
    \end{align*}

    Define $n_i = n_{y_i}$ and $V_i = \mathcal{U}_1^{(y_i)}$ for all $i = 1, \dots, N$. Note that $V_1, \dots, V_N \subset X_\varphi$ by property (c). Since $\sigma$ is injective on $X_\varphi$, and since $W_1, \dots, W_N$ are disjoint, then $V_1, \dots, V_N$ are disjoint subsets. Moreover,
    \begin{align*}
        X_\varphi
        = \sigma^{-1}(Y) \cap X_\varphi
        = \sigma^{-1} \left( \bigcup_{i=1}^N W_i \right) \cap X_\varphi
        = \bigcup_{i=1}^N \sigma^{-1}(W_i) \cap X_\varphi
        = \bigcup_{i=1}^N V_i.
    \end{align*}
    Now we use the sets above to define the map $\mu$ from $C^*(R(\sigma))$ to  $\bigoplus_{i=1}^N C_0(V_i, M_{n_i})$ by
    \begin{align}
        \label{eqn:muCRsigma}
        \mu(f)(x)_{rs} = f(\sigma\vert_{\mathcal{U}_r^{(y_i)}}^{-1}(\sigma(x)), \sigma\vert_{\mathcal{U}_s^{(y_i)}}^{-1}(\sigma(x))),
    \end{align}
    for $x \in V_i$, $r, s = 1, \dots, n_i$, $i =1, \dots, n_i$, $f \in C^*(R(\sigma))$.

    It is straightforward that $\mu$ is well-defined. We show that $\mu$ is an $\ast$-isomorphism.

    \begin{itemize}
        \item $\mu$ is a homomorphism

        In order to simplify our notation, given $i = 1,\dots, N$, $x \in \sigma^{-1}(W_i)$, and $r = 1, \dots, n_i$, we denote $x_r = \sigma\vert_{\mathcal{U}_r^{(y_i)}}^{-1}(\sigma(x))$. Note that $\sigma^{-1}(\lbrace \sigma(x) \rbrace) = \lbrace x_1, \dots, x_{n_i} \rbrace.$ Let $f, f_1, f_2 \in C^*(R(\sigma))$ and let $x \in V_i$ with $i = 1, \dots, N$. Then, for $r, s = 1, \dots, n_i$,
        \begin{align*}
            \mu(f^*)(x)_{rs}
            = f^*(x_r, x_s) = \overline{f(x_s, x_r)} 
            = (\mu(f)(x)^*)_{rs}.
        \end{align*}
        By properties (a) and (b), we have that an element $u \in X$ satisfies $\sigma(u) = \sigma(x)$ if, and only if, $u = x_j$ for some $j = 1, \dots, n_i$. Then
        \begin{align*}
            \mu(f_1 \cdot f_2)(x)_{rs}
            = f_1 \cdot f_2 (x_r, x_s)
            = \sum_{t = 1}^{n_i} f_1(x_r, x_t) f_2(x_t, x_s)
            = (\mu(f_1) \mu(f_2))(x)_{rs}.
        \end{align*}
        Thus $\mu$ is a homomorphism.

        \item $\mu$ is an isometry

        By \cite[Proposition 9.3.1]{sims2017hausdorff}, for all $x \in X$, the $\ast$-representation $\pi_x: C_c(R(\sigma)) \rightarrow B(\ell^2(R(\sigma)_x))$ is given by
        \begin{align*}
            \pi_x(f)\delta_{(v,x)} = \sum_{(u,v) \in R(\sigma)} f(u,v) \delta_{(u,x)}
            = \sum_{u: \sigma(u) = \sigma(x)} f(u,v) \delta_{(u,x)},
        \end{align*}
        for $f \in C_c(R(\sigma))$ and $v \in X$ with $\sigma(v) = \sigma(x)$.

        Let $x \in X_\varphi$. Then there exists a unique $i = 1, \dots, N$ such that $x \in V_i$. To simplify our notation, we let $x_r = \sigma_{\mathcal{U}_r^{(y_i)}}^{-1}(\sigma(x))$ for all $r = 1, \dots, n_i$. Note that $\sigma^{-1}(\sigma(x)) = \lbrace x_1, \dots, x_{n_i} \rbrace$.

        Let $T_x: \mathbb{C}^{n_i} \rightarrow \ell^2(R(\sigma)_x)$ be the isometry defined by $T_x e_r = \delta_{(x_r, x)}$ for $r = 1,\dots,n_i$. Then, for all $s = 1, \dots, n_i$, we have
        \begin{align*}
            \pi_x(f) T_x e_s
            &= \pi_x(f) \delta_{(x_s, x)} \\
            &= \sum_{u: \sigma(u) = \sigma(x)} f(u, x_s) \delta_{(u, x)} \\
            &= \sum_{r=1}^{n_i} f(x_r, x_s) \delta_{(x_r, x)} \\
            &= T_x \sum_{r=1}^{n_1} f(x_r, x_s) e_r \\
            &=T_x \mu(f)(x) e_s.
        \end{align*}
        This implies that $\Vert \pi_x(f) \Vert = \Vert \mu(f)(x) \Vert$ for all $x \in X_\varphi$.

        Given $x \in X$, let $u = \varphi(\sigma(x)) \in X_\varphi$. \cite[Proposition 9.3.1]{sims2017hausdorff} implies that there exists a unitary $U_{(x, u)}: \ell^2(R(\sigma)_u) \rightarrow  \ell^2(R(\sigma)_x)$ such that $\pi_x = U_{(x,u)} \pi_u U_{(x,u)}^*$. Then $\Vert \pi_x(f) \Vert = \Vert \pi_u(f) \Vert$. Therefore,
        \begin{align*}
            \Vert f \Vert
            = \sup_{x \in X} \Vert \pi_x(f) \Vert
            = \sup_{x \in X_\varphi} \Vert \pi_x(f) \Vert
            = \sup_{x \in X_\varphi} \Vert \mu(f)(x) \Vert
            = \Vert \mu(f) \Vert.
        \end{align*}
        Therefore the restriction of $\mu$ to $C_c(R(\sigma))$ is an isometry. By continuity, $\mu$ is an isometry.

        \item $\mu$ is an isomoprism

        Since $\mu$ is an isometry, it follows that it is injective. We show that $\mu$ is surjective. Let $F \in \bigoplus_{i=1}^N C_c(V_i, M_{n_i})$. Define $f \in C_c(R(\sigma))$ by
        \begin{align*}
            f(x_r, x_s)
            = F(x)_{rs}
            \hspace{15pt}
            \text{for $x \in V_i$, $r,s = 1, \dots, n_i$.}
        \end{align*}
        Note that $f$ is continuous because it is the composition of continuous functions. Let $K \subset X_\varphi$ by a compact subset containing the support of $F$, then the support of $f$ is a subset of $[\sigma^{-1}(\sigma(K)) \times \sigma^{-1}(\sigma(K))] \cap R(\sigma)$, which is compact. This implies that $f \in C_c(R(\sigma))$. Moreover, $\mu(f) = F$. So far we have that the restriction
        \begin{align*}
            \mu\vert_{C_c(R(\sigma))}: C_c(R(\sigma)) \rightarrow \bigoplus_{i=1}^N C_c(V_i, M_{n_i})
        \end{align*}
        is surjective. Since $\bigoplus_{i=1}^N C_c(V_i, M_{n_i})$ is dense in $\bigoplus_{i=1}^N C_0(V_i, M_{n_i})$ and since $\mu$ is an isometry, then $\mu$ is surjective.
    \end{itemize}

    Now we show property (i). Let $f \in C_c(R(\sigma))$, and $x \in V_i$ with $i = 1, \dots, N$. Since $\sigma^{-1}(x) = \lbrace x_1, \dots, x_{n_i} \rbrace$, then
    \begin{align*}
        \mathrm{tr}\hphantom{.}\mu(f)(x)
        = \sum_{l=1}^{n_i} f(x_l, x_l)
        = \sum_{u: \sigma(u) = \sigma(x)} f(u).
    \end{align*}

    Now we prove property (ii). Let $k$ be a positive integer, and let $F = [F_{rs}] \in M_k(C^*(R(\sigma)))$, and fix $x \in V_i$ with $i = 1, \dots, N$. Then
    \begin{align*}
        \mathrm{tr}\hphantom{.}\mu_k(F)(x)
        = \sum_{q = 1}^{k n_i} \mu_k(F)_{qq}(x)
        = \sum_{j= 0}^{k-1} \sum_{l=1}^{n_i} \mu_k (F)_{n_ij + l, n_ij + l}(x)
        = \sum_{i=1}^k \sum_{l=1}^{n_i} \mu(F_{ii})_{ll}(x).
    \end{align*}
    Thus property (ii) holds.

    Next, we prove item (iii). Let $k$ be a positive integer, and fix $p \in \mathcal{P}_k(C^*(R(\sigma)))$. Note that $C^*(R(\sigma)) \subset C(R(\sigma))$, so it follows that, for all $i = 1, \dots, k$, $p_{ii} \in C(R(\sigma))$. given $x \in X_\varphi$, items (i) and (ii) imply that
    \begin{align*}
        \mathrm{tr}\hphantom{.}\mu_k(p)(x)
        = \sum_{i=1}^k \mathrm{tr}\hphantom{.}\mu(p_{ii})(x)
        = \sum_{i=1}^k \sum_{u: \sigma(u) = \sigma(x)} p_{ii}(u).
    \end{align*}
    The equation above implies that $\mathrm{tr}\hphantom{.}\mu_k(p)$ is a composition of continuous functions, then $\mathrm{tr}\hphantom{.}\mu_k(p)$ is a continuous function. It follows from Lemma \ref{lemma:traceCXMn} that $\mathrm{tr}\hphantom{.}\mu_k(p) \in C_c(X_\varphi, \mathbb{N})$.

    Now we show item (iv). Let $x \in X_\varphi$. By definition of $\mu$,
\begin{align*}
\mu_{k+l}(p \oplus q)(x)
= \mu_k(p)(x) \oplus \mu_l(q)(x).
\end{align*}
Then item (ii) implies that $\mathrm{tr}\hphantom{.}\mu(p\oplus q)(x) = \mathrm{tr}\hphantom{.}\mu_k(p)(x) + \mathrm{tr}\hphantom{.}\mu_k(q)(x).$

Next, we show (v). Suppose $p \sim q$. Then $p \oplus 0_l,$ $q \oplus 0_k \in \mathcal{P}_l(C^*(R(\sigma)))$ are equivalent. Given $x \in X_\varphi$, then $\mu_{k+l}(p \oplus 0_l), \mu_{k+l}(q \oplus 0_k)$ are equivalent in $M_{k+l}$. This  fact and item (iv) imply that
\begin{align*}
\mathrm{tr}\hphantom{.}\mu_k(p)(x)
= \mathrm{tr}\hphantom{.}\mu_{k+l}(p \oplus 0_l)(x)
= \mathrm{tr}\hphantom{.}\mu_{k+l}(q \oplus 0_k)(x)
= \mathrm{tr}\hphantom{.}\mu_l(q)(x).
\end{align*}
Conversely, assume that $\mathrm{tr}\hphantom{.}\mu_k(p) = \mathrm{tr}\hphantom{.} \mu_l(q)$. By item (iv), $\mathrm{tr}\hphantom{.}\mu_{k+l}(p \oplus 0_l) = \mathrm{tr}\hphantom{.} \mu_{k+l}(q \oplus 0_k)$. Hence, by Proposition \ref{lemma:traceCXMn}, $\mu_{k+l}(p \oplus 0_l) \sim \mu_{k+l}(q \oplus_k)$. Since $\mu_{k+l}$ is an isomorphism, we have $p \oplus 0_l \sim q \oplus 0_k$ and then $p \sim q$.

Finally, we prove item (vi). Suppose $\mathrm{tr}\hphantom{.}\mu_k(p) = 0$. Then, for all $x \in X_\varphi$, $0 = \mathrm{tr}\hphantom{.}\mu_k(p)(x)
= \dim \mu_k(p(x)) \mathbb{C}^k.$ So $\mu_k(p(x)) = 0$. Since $x$ is arbitrary and $\mu_k$ is an isomorphism, we have $p = 0$.
\end{proof}

We use the isomorphism $\mu$ to define a trace operator on $\mathcal{P}_\infty(C^*(R(\sigma)))$ that characterises the $K_0$-group of $C^*(R(\sigma))$. Given a continuous section $\varphi: Y \rightarrow X$ of $\sigma$, we define $\mathrm{tr}_\varphi: \mathcal{P}_\infty(C^*(R(\sigma))) \rightarrow C(X, \mathbb{N})$ by
\begin{align*}
\mathrm{tr}_\varphi \hphantom{.} p(x)
= \begin{cases}
\mathrm{tr}\hphantom{.}\mu_k(p)(x) & \text{if } x \in X_\varphi, \\
0 & \text{otherwise,}
\end{cases}
\end{align*}
for $p \in \mathcal{P}_k(C^*(R(\sigma)))$, $k \geq 1$, and $x \in X$. Note that $\mathrm{tr}_\varphi \hphantom{.} p$ is the extension of $\mathrm{tr}\hphantom{.}\mu_k(p)$ to a larger domain $X \supset X_\varphi$. We do this because later we define the isomorphism from $K_0(C^*(R(\sigma)))$ to $H_0(R(\sigma))$ by $[p]_0 \mapsto [\mathrm{tr}_\varphi\hphantom{.}p]$ for $p \in \mathcal{P}_\infty(C^*(R(\sigma)))$. In order to study the equivalence class of $\mathrm{tr}_\varphi \hphantom{.} p$ in the homology group of $H_0(R(\sigma))$, the function $\mathrm{tr}_\varphi \hphantom{.} p$ needs to be defined in the unit space $X$ of $R(\sigma)$.

\begin{remark}
    Note that, if the projection $p$ is in $C^*(R(\sigma))$, then for an arbitrary continuous section $\varphi$ of $\sigma$, $x \in X$, and $f \in C^*(R(\sigma))$
    \begin{align*}
        \mathrm{tr}_\varphi \hphantom{.}p(x) = 1_{X_\varphi}(x) \sum_{u: \sigma(u) = \sigma(x)} f(u).
    \end{align*}
\end{remark}

The following lemma shows the connection of the trace operator $\mathrm{tr}_\varphi$ with the homology group $H_0(R(\sigma))$.

\begin{lemma}
\label{lemma:traceH0}
Let $\varphi, \psi$ be two continuous sections of $\sigma$. Then
\begin{enumerate}[(i)]
\item Given $p \in \mathcal{P}_\infty(C^*(R(\sigma)))$, $\mathrm{tr}_\varphi\hphantom{.}p \sim \mathrm{tr}_\psi\hphantom{.}p$.
\item For a compact open set $V \subset X$, we have $1_V \sim \mathrm{tr}_\varphi 1_V$.
\end{enumerate}
\end{lemma}
\begin{proof}
    We begin by proving (i). Fix $p \in \mathcal{P}_\infty(C^*(R(\sigma)))$, then there exists a positive integer $k$ with $p \in \mathcal{P}_k(C^*(R(\sigma)))$. Given $x \in X$, set $x_\varphi = \varphi(\sigma(x))$ and $x_\psi = \psi(\sigma(x))$. Then $x_\varphi$ (resp. $x_\psi$) is the only element in $X_\varphi$ (resp. $X_\psi$) with $\sigma(x_\varphi) = \sigma(x_\psi) = \sigma(x)$. By definition of $\mathrm{tr}_\varphi\hphantom{.}p$ and by Proposition \ref{prop:isoRsigmacompact}, we have
    \begin{align*}
        \sum_{u: \sigma(u) = \sigma(x)} \mathrm{tr}_\varphi \hphantom{.} p(u)
        &= \mathrm{tr}_\varphi \hphantom{.} p(x_\varphi)
        = \sum_{i=1}^k \mathrm{tr}_\varphi \hphantom{.} p_{ii} (x_\varphi)
        = \sum_{i=1}^k \sum_{v: \sigma(v) = \sigma(x_\varphi)} p_{ii}(v) \\
        &= \sum_{i=1}^k \sum_{v: \sigma(v) = \sigma(x_\psi)} p_{ii}(v) 
        = \mathrm{tr}_\psi\hphantom{.} p(x_\psi)
        = \sum_{u: \sigma(u) = \sigma(x)} \mathrm{tr}_\psi \hphantom{.} p(u).
    \end{align*}
    Since $x$ is arbitrary, it follows from Lemma \ref{lemma:traceH0Rsigma} that $\mathrm{tr}_\varphi\hphantom{.}p \sim \mathrm{tr}_\psi\hphantom{.}p$.

    Now we prove item (ii). Let $V \subset X$ a compact open subset, let $x \in X$, and set $x_\varphi = \varphi(\sigma(x))$. Then
    \begin{align*}
        \sum_{u: \sigma(u) = \sigma(x)}\mathrm{tr}_\varphi\hphantom{.}1_V(u)
        = \mathrm{tr}_\varphi \hphantom{.} 1_V (x_\varphi)
        = \sum_{u: \sigma(u) = \sigma(x_\varphi)} 1_V(u)
        = \sum_{u: \sigma(u) = \sigma(x)} 1_V(u).
    \end{align*}
    Since $x$ is arbitrary, Lemma \ref{lemma:traceH0Rsigma} implies that $\mathrm{tr}_\varphi \hphantom{.} 1_V \sim 1_V$.
\end{proof}

The following theorem gives an isomorphism of the ordered groups $K_0(C^*(G))$  and $H_0(G)$. Corollary 5.2 by Farsi, Kumjian, Pask and Sims \cite{FKPS} and Matui's \cite[Theorem 4.10]{Matui} prove isomorphisms of homology groups of groupoids and the corresponding $K_0$-group. However, \cite[Corollary 5.2]{FKPS} does not prove whether such isomorphism preseves positive elements from one group to another, and Matui only considers groupoids with compact unit spaces, without providing a detailed proof. Our theorem generalises both results, and we describe the isomorphism by an explicit formula.

\begin{theorem}
    \label{thm:K0H0}
    Let $G$ be an AF groupoid, written as the inductive limit $G = \displaystyle\lim_\rightarrow G_n$ of elementary groupoids. Then, for every $n$, there exists a locally compact, Hausdorff, second countable, totally disconntect space $Y_n$, and there is a surjective local homeomorphism $\sigma_n: G_n^{(0)} \rightarrow Y_n$ such that $G_n = R(\sigma_n)$. Moreover,
    \begin{enumerate}
        \item there exists a unique isomorphism $\nu: K_0(C^*(G)) \rightarrow H_0(G)$ defined by
        \begin{align*}
            \nu([p]_0) = [\mathrm{tr}_{\varphi_n}(p)]
        \end{align*}
        for $p \in \mathcal{P}_n(C^*(G_n))$, and such that $\varphi_n$ is an arbitrary continuous section of $\sigma_n$;
        \item the map $\nu$ is an isomorphism of ordered groups, i.e., $\nu(K_0(C^*(G))^+) = H_0(G)^+$; and
        \item the isomorphism $\nu$ is precisely the map given by \cite[Corollary 5.2]{FKPS}. In other words, $\nu([1_V]_0) = [1_V]$ for all $V \subset G^{(0)}$ compact open.
    \end{enumerate}
\end{theorem}

We divide the proof into three parts:
\begin{enumerate}[{Part} 1]
\item First we study the groupoid $R(\sigma)$ where $\sigma$ satisfies the conditions of the beginning of the chapter, i.e., that there exists a compact subset $X_0$ such that $\sigma^{-1}(\sigma(X_0)) =X_0$ and that $\sigma$ is injective on $X \setminus X_0$.
\item Then we generalise the result to a groupoid of the form $R(\sigma)$, where $X$ no longer need to necessarily satisfy the conditions of Part 1.
\item Finally, we prove the isomorphism for an AF groupoid $G$.
\end{enumerate}

\begin{proof}[Proof of Part 1]\renewcommand{\qedsymbol}{}
    Fix a continuous section $\varphi$ of $\sigma$. Given $p \in \mathcal{P}_\infty(C^*(R(\sigma)))$, let $\nu([p]_0) = [\mathrm{tr}_\varphi \hphantom{.}p]$. It follows from item (v) of Proposition \ref{prop:isoRsigmacompact} that $\nu([p]_0)$ is well-defined. Moreover, item (iv) of the same proposition implies that for a projection $q \in \mathcal{P}_\infty(C^*(R(\sigma)))$,
    \begin{align*}
        \nu([p]_0+[q]_0)
        = \nu([p\oplus q]_0)
        = [\mathrm{tr}_\varphi(p \oplus q)]
        %= [\mathrm{tr}_\varphi(p) + \mathrm{tr}_\varphi(q)]
        = [\mathrm{tr}_\varphi(p)]+[\mathrm{tr}_\varphi(q)]
        = \nu([p]_0) + \nu([q]_0).
    \end{align*}
    Then we extend $\nu$ uniquely to a group homomorphism $\nu: K_0(C^*(R(\sigma))) \rightarrow H_0(R(\sigma))$.

    We claim that $\nu$ is an isomorphism. Let $x \in K_0(C^*(R(\sigma)))$ be such that $\nu(x) = 0$. Then there are $p, q \in \mathcal{P}_\infty(C^*(R(\sigma)))$ such that $x = [p]_0 - [q]_0$. Then
    \begin{align*}
        [\mathrm{tr}_\varphi \hphantom{.} p] = \nu([p]_0) = \nu([q]_0) = [\mathrm{tr}_\varphi \hphantom{.} q],
    \end{align*}
    by definition of the homomorphism $\nu$. It follows from Lemma \ref{lemma:traceH0Rsigma} that, for every $x \in X$,
    \begin{align*}
        \sum_{u:\sigma(u) = \sigma(x)} \mathrm{tr}_\varphi \hphantom{.} p(u) = \sum_{u:\sigma(u) = \sigma(x)} \mathrm{tr}_\varphi \hphantom{.} q(u).
    \end{align*}
    This equivalence implies that
    \begin{align*}
        \mathrm{tr}_\varphi \hphantom{.} p(\varphi(\sigma(x))) = \mathrm{tr}_\varphi \hphantom{.} q(\varphi(\sigma(x))).
    \end{align*}
    Since $\mathrm{tr}_\varphi \hphantom{.} p$ and $\mathrm{tr}_\varphi \hphantom{.} q$ have support in $X_\varphi$, then $\mathrm{tr}_\varphi \hphantom{.} p = \mathrm{tr}_\varphi \hphantom{.} q$. It follows from item (v) of Proposition \ref{prop:isoRsigmacompact} that $p \sim q$. Hence, $x = 0$ and therefore $\nu$ is injective. We show that $\nu$ is also surjective. Given a compact open subset $V \subset X$, Lemma \ref{lemma:traceH0}  implies that
    \begin{align*}
        [1_V] = [\mathrm{tr}_\varphi\hphantom{.} 1_V] = \nu([1_V]_0).
    \end{align*}
    Since the elements of the form $[1_V]$ generate $H_0(R(\sigma))$, it follows that $\nu$ is a group isomorphism.

    Next, we prove that $\nu$ is an isomorphism of ordered groups. Item (iii) of Proposition \ref{prop:isoRsigmacompact} implies that $\nu(K_0(C^*(R(\sigma))^+) \subset H_0(R(\sigma))^+$. Let $f \in C_c(X, \mathbb{N})$. Then there are compact open sets $V_1, \dots, V_n \subset X$, and non-negative integers $a_1, \dots, a_n$  such that $f = a_1 1_{V_1} + \dots + a_n 1_{V_n}$. Then
    \begin{align*}
        [f]
        = \sum_{i=1}^n a_i [1_{V_i}]
        = \sum_{i=1}^n a_i \nu([1_{V_i}]_0)
        = \nu \left(\sum_{i=1}^n a_i [1_{V_i}]_0 \right).
    \end{align*}
    Then $\nu(K_0(C^*(R(\sigma)))^+) = H_0(R(\sigma))^+$. This completes the proof of Part 1.
\end{proof}

We need the following lemma before proving Part 2. Now we no longer assume that $\sigma$ satisfies that conditions of Part 1.

\begin{lemma}
\label{lemma:tracevarphi}
Let $\sigma: X \rightarrow Y$ be a surjective local homeomorphism with continuous section $\varphi: Y \rightarrow X$. Let $X = \bigcup_{k=1}^\infty X_k$ be an increasing union of compact open sets such that $\varphi(\sigma(X_k)) \subset X_k$. For each $k$, define the map $\sigma_k : X \rightarrow \sigma(X_k) \sqcup (X \setminus X_k)$ by
\begin{align*}
    \sigma_k(x) =
    \begin{cases}
        \sigma(x) & \text{if }x \in X_k, \\
        x & \text{otherwise.}
    \end{cases}
\end{align*}
Define also the map $\varphi_k: \sigma(X_k) \sqcup (X \setminus X_k) \rightarrow X$ by
\begin{align*}
    \varphi_k(y) =
    \begin{cases}
        \varphi(y) & \text{if } y \in \sigma(X_k), \\
        y & \text{otherwise.}
    \end{cases}
\end{align*}
Then
\begin{enumerate}[(i)]
    \item $\sigma_k$ is a surjective local homeomorphism with continuous section $\varphi_k$.
    \item $R(\sigma)$ is an inductive limit of the form $R(\sigma) = \displaystyle\lim_\rightarrow R(\sigma_k)$.
    \item For $k \geq 1$, $p \in \mathcal{P}_\infty(C^*(R(\sigma_k))$, we have $\mathrm{tr}_\varphi \hphantom{.} p = \mathrm{tr}_{\varphi_l} \hphantom{.} p$ for all $l \geq k$, where $\mathrm{tr}_\varphi \hphantom{.} p$ is given by
    \begin{align*}
        \mathrm{tr}_\varphi \hphantom{.} p (x) = 1_{X_\varphi}(x) \sum_{i=1}^n \sum_{u : \sigma(u) = \sigma(x)}  p_{ii}(u),
        \hspace{15pt}
        \text{for $x \in X$.}
    \end{align*}
\end{enumerate}
\end{lemma}
\begin{proof}
    Since $\sigma$ is a surjective local homeomorphism and $X_k$ is open, then the restriction $\sigma_k\vert_{X_k}: X_k \rightarrow \sigma(X_k)$ is a surjective local homeormorphism. The restriction $\sigma_k$ on $X \setminus X_k$ is the identity. Thus, $\sigma_k$ is a surjective local homeomorphism. It is straightforward that $\varphi_k$ is a continuous section of $\sigma_k$. This proves (i).

    Next we prove (ii). Let $(x,y) \in R(\sigma)$. Since $X$ is covered by the union of the sets $X_k$, then there exists a $k$ such that $x, y \in X_k$. By definition of $\sigma_k$, we have that $\sigma_k(x) = \sigma_k(y)$, and then $(x, y) \in R(\sigma_k)$. Therefore, $R(\sigma) \subset \bigcup_{k=1}^\infty R(\sigma_k)$. By similar arguments, we have that $\bigcup_{k=1}^\infty R(\sigma_k) \subset R(\sigma)$. It is straightforward that $R(\sigma_k)$ is an open subset of $R(\sigma_{k+1})$ for all $k$. Therefore, (ii) holds.

    Finally, we show (iii). Let $k \geq 1$, and let $p \in \mathcal{P}_\infty(C_c(R(\sigma_k)))$. Then there exists an $n$ such that $p \in \mathcal{P}_n(C_c(R(\sigma)))$. By increasing $n$ if necessary, we have $p \in \mathcal{P}_n(C_c(R(\sigma_n)))$ and that the support of $p$ is included in $X_n$. Then for any $l \geq n$ and $x \in X$
    \begin{align*}
        \mathrm{tr}_{\varphi_l}\hphantom{.}p(x)
        = 1_{X_{\varphi_l}}(x) \sum_{i=1}^n \sum_{u: \sigma_l(u) = \sigma_l(x)} p_{ii}(u)
        = 1_{X_{\varphi_l}}(x) \sum_{i=1}^n \sum_{u \in X_l: \sigma_l(u) = \sigma_l(x)} p_{ii}(u)
        %&= 1_{X_{\varphi_l}}(x) \sum_{i=1}^n 1_{X_l} \sum_{u : \sigma(u) = \sigma(x)} p_{ii}(u) \\
        %= 1_{X_{\varphi_l} \cap X_l}(x) \sum_{i=1}^n \sum_{u : \sigma(u) = \sigma(x)} p_{ii}(u),
    \end{align*}
    while
    \begin{align*}
        \mathrm{tr}_{\varphi}\hphantom{.}p(x)
        = 1_{X_\varphi}(x) \sum_{i=1}^n \sum_{u:\sigma(u) = \sigma(x)} p_{ii}(u)
        = 1_{X_\varphi}(x) \sum_{i=1}^n \sum_{u \in X_l:\sigma(u) = \sigma(x)} p_{ii}(u)
        %= 1_{X_\varphi}(x) \sum_{i=1}^n 1_{\sigma^{-1}(\sigma(X_l))}(x) \sum_{u \in X_l:\sigma(u) = \sigma(x)} p_{ii}(u)
        %= 1_{X_\varphi \cap \sigma^{-1}(\sigma(X_l))}(x) \sum_{i=1}^n \sum_{u \in X_l:\sigma(u) = \sigma(x)} p_{ii}(u).
    \end{align*}
    Therefore $\mathrm{tr}_{\varphi}\hphantom{.}p(x) = \mathrm{tr}_{\varphi_l}\hphantom{.}p (x)$.

    Now let $p \in \mathcal{P}_n(C^*(R(\sigma)))$. Since $C^*(R(\sigma)) = \displaystyle\lim_\rightarrow C^*(R(\sigma_m))$, then $K_0(C^*(R(\sigma))) = \displaystyle\lim_\rightarrow K_0(C^*(R(\sigma_m)))$ as ordered groups. Then there is $q \in \mathcal{P}_\infty(C^*(R(\sigma_m)))$ with $p \sim q$. By Part 1 and by item (v) of Proposition \ref{prop:isoRsigmacompact}, we have $\mathrm{tr}_{\varphi_m} \hphantom{.} p = \mathrm{tr}_{\varphi_m} \hphantom{.} q$. By applying the formula of item 3 for the isomorphism $K_0(C^*(R(\sigma_m))) \cong H_0(R(\sigma_m))$, we can assume that $q \in C_c(R(\sigma_m))$ without loss of generality. Then $q \in C_c(R(\sigma))$. Hence, there exists a $k$ such that $\mathrm{tr}_\varphi\hphantom{.}q = \mathrm{tr}_{\varphi_l}\hphantom{.}q$ for all $l \geq k$.  Then
    \begin{align}
        \mathrm{tr}_{\varphi}\hphantom{.}q
        = \mathrm{tr}_{\varphi_l}\hphantom{.}q
        = \mathrm{tr}_{\varphi_l}\hphantom{.}p.
    \end{align}
    
    Let $x \in X$. Choose $j \geq l$ such that $X_j$ contains both $x$ and $\sigma(x)$. Then
    \[
        \pushQED{\qed} 
        \mathrm{tr}_{\varphi}\hphantom{.}p(x)
        %= 1_{X_\varphi}(x) \sum_{i=1}^n \sum_{u:\sigma(u)=\sigma(x)} p_{ii}(u)
        = 1_{X_\varphi}(x) \sum_{i=1}^n \sum_{u:\sigma_j(u)=\sigma_j(x)} p_{ii}(u)
        = \mathrm{tr}_{\varphi_j}\hphantom{.}p(x)
        = \mathrm{tr}_{\varphi_l}\hphantom{.}p(x). \qedhere
        \popQED
    \]
    \renewcommand{\qedsymbol}{}
\end{proof}

\begin{proof}[Proof of Part 2]\renewcommand{\qedsymbol}{}
    Fix a continuous section $\varphi$ of $\sigma$. Since $X$ is second countable and totally disconnected, then we write $X = \bigcup_{k=1}^\infty \widetilde{X}_k$ as the increasing union of compact open sets $\widetilde{X}_k$. For each $k$, let $X_k = \widetilde{X_k} \cap \varphi(\sigma(X_k))$. Then $X = \bigcup_{k=1}^\infty X_k$ is the increasing union of the compact open sets $X_k$, with $\varphi(\sigma(X_k)) \subset X_k$. Lemma \ref{lemma:tracevarphi} gives a sequence of surjective local homeomorphisms $\sigma_k: X_k \rightarrow \sigma(X_k) \sqcup (X \setminus X_k)$ satisfying the properties of the lemma.

    By Part 1, we have the isomorphism $\nu_k: K_0(C^*(R(\sigma_k))) \rightarrow H_0(R(\sigma_k))$ of ordered groups given by $\nu_k([p]_{0,k}) = [\mathrm{tr}_{\varphi_k}\hphantom{.}p]_k$ for $p \in \mathcal{P}_\infty(C^*(R(\sigma_k)))$. This isomorphism is such that $[1_V]_{0,k} \mapsto [1_V]_k$ for all compact open sets $V \subset X$.

    Denote $[\cdot]_{0,k}$, $[\cdot]_{0}$, $[\cdot]_k$ and $[\cdot]$ the equivalence classes in the groups
    \begin{align*}
        K_0(C^*(R(\sigma_k))),\hspace{10pt}
        K_0(C^*(R(\sigma))), \hspace{10pt}
        H_0(R(\sigma_k))\hspace{10pt}
        \text{ and }\hspace{10pt}
        H_0(R(\sigma)),
    \end{align*}
    respectively. Recall that $K_0(C^*(R(\sigma))) = \displaystyle\lim_\rightarrow K_0(C^*(R(\sigma_k))$ and $H_0(R(\sigma)) = \displaystyle\lim_\rightarrow H_0(R(\sigma_k))$. For $k \geq 1$, let
    \begin{align*}
        i_k: K_0(C^*(R(\sigma_k))) \rightarrow K_0(C^*(R(\sigma_{k+1})))
        \hspace{5pt}
        \text{and}
        \hspace{5pt}
        j_k: H_0(R(\sigma_k)) \rightarrow H_0(R(\sigma_{k+1}))
    \end{align*}
    be the connecting morphisms, and let $\alpha_k: K_0(C^*(R(\sigma_k))) \rightarrow K_0(C^*(R(\sigma)))$ and $\beta_k: H_0(R(\sigma_k)) \rightarrow H_0(R(\sigma))$ be the inclusion morphisms such that
    \begin{align*}
        i_k([p]_{0,k}) = [p]_{0,k+1}
        \hspace{10pt}\text{and}\hspace{10pt}
        \alpha_k([p]_{0,k}) = [p]_0
        \hspace{10pt}
        \text{for }
        p \in \mathcal{P}_\infty(C^*(R(\sigma_k))), \\
        j_k([f]_{k}) = [f]_{k+1}
        \hspace{13pt}\text{and}\hspace{17pt}
        \beta_k([f]_k) = [f]
        \hspace{15pt}
        \text{for }
        f \in C(X, \mathbb{Z}).
        \hspace{40pt}
    \end{align*}
    The morphisms $i_k$, $j_k$, $\alpha_k$, $\beta_k$ preserve positive elements from one group to another. It is straightforward that $j_k \circ \nu_k([1_V]_{0,k}) = \nu_{k+1} \circ i_k([1_V]_{0,k})$ for the elements $V$ compact open. Since the elements $[1_V]_{0,k}$ generate $K_0(C^*(R(\sigma_k)))$, we have $j_k \circ \nu_k= \nu_{k+1} \circ i_k$. Lemma \ref{lemma:homomorphisminductive} gives an isomorphism $\nu: K_0(C^*(R(\sigma))) \rightarrow H_0(R(\sigma))$ that makes the diagram below commutative for all $k$.
    \begin{equation*}
        \begin{tikzcd}
            K_0(C^*(R(\sigma_k)) \arrow[r,"\nu_k"] 
            \arrow{d}{}[swap]{\alpha_k}
            & H_0(R(\sigma_k)) \arrow[d, "\beta_k"]\\
            K_0(C^*(R(\sigma)) \arrow{r}{\nu}[swap]{}
            & H_0(R(\sigma))
        \end{tikzcd}
    \end{equation*}
    Then, for all compact open subset $V \subset X$ we have that
    \begin{align*}
        \nu([1_V]_0) = \nu \circ \alpha_k([1_V]_{0,k})
        = \beta_k \circ \nu_k([1_V]_{0,k})
        = \beta_k([1_V]_k)
        = [1_V].
    \end{align*}

    Now let $n \geq 1$ and $p \in \mathcal{P}_n(C^*(R(\sigma))) \subset K_0(C^*(R(\sigma)))^+$. By Lemma \ref{lemma:tracevarphi}, there exists a $k$ such that $\mathrm{tr}_{\varphi_k}\hphantom{.}p = \mathrm{tr}_{\varphi}\hphantom{.}p$. It follows from Part 1 that
    \begin{align}
        \nu([p]_0)
        = \nu(\alpha_k([p]_{0,k})
        = \beta_k \circ \nu_k([p]_{0,k})
        = \beta_k([\mathrm{tr}_{\varphi_k}\hphantom{.}p]_k)
        = [\mathrm{tr}_{\varphi_k}\hphantom{.}p]
        = [\mathrm{tr}_{\varphi}\hphantom{.}p]. \label{eqn:nutrvarphi}
    \end{align}
    This completes the proof of Part 2.
\end{proof}

\begin{proof}[Proof of Part 3]
    The proof is analogous to Part 2. Let $G$ be an AF groupoid. Write $G = \displaystyle\lim_\rightarrow R(\sigma_n)$, where $\sigma: G^{(0)} \rightarrow Y_n$ is a surjective local homeomorphism of totally disconnected spaces. As in Part 2, let $i_n: K_0(C^*(R(\sigma_n))) \rightarrow K_0(C^*(R(\sigma_{n+1})))$ and $j_n: H_0(R(\sigma_n)) \rightarrow H_0(R(\sigma_{n+1}))$ be the inclusion maps. By Part 2, we have an isomorphism $\nu_n: K_0(C^*(R(\sigma_n))) \rightarrow H_0(R(\sigma_n))$ satisfying the properties 1--3 of this theorem for all $n$. Lemma \ref{lemma:homomorphisminductive} gives an isomorphism $\nu: K_0(C^*(G)) \rightarrow H_0(G)$ of ordered groups such that the we have the following commutative diagram for all $n$, where $\alpha_n$ and $\beta_n$ are the inclusion maps.
    \begin{equation*}
        \begin{tikzcd}
        K_0(C^*(R(\sigma_n)) \arrow[r,"\nu_n"] 
        \arrow{d}{}[swap]{\alpha_n}
        & H_0(R(\sigma_n)) \arrow[d, "\beta_n"]\\
        K_0(C^*(G) \arrow{r}{\nu}[swap]{}
        & H_0(G)
        \end{tikzcd}
    \end{equation*}
    Using the same ideas to prove \eqref{eqn:nutrvarphi} in Part 2, we have that $\nu$ satisfies properties 1--3. This completes the proof of the theorem.
\end{proof}

\begin{corollary}
\label{corollary:1VgenerateK0}
Given an AF groupoid $G$, $K_0(C^*(G))^+$ is generated  by elements of the form $[1_V]_0$, where $V \subset G^{(0)}$ are compact open subsets.
\end{corollary}
\begin{proof}
    Recall that $H_0(G)^+ = \lbrace [f] : f \in C_c(X, \mathbb{N}) \rbrace$. Since $G^{(0)}$ is locally compact, Hausdorff, second countable, and totally disconnected, then every $f \in C_c(G^{(0)}, \mathbb{N})$ is a linear combination of the form $f = \alpha_1 1_{V_1} + \dots + \alpha_n 1_{V_n}$, where all $\alpha_i$ non-negative integers, and $V_i$ are compact open. This implies that the classes of the form $[1_V]$, for $V \subset G^{(0)}$ compact open, generate $H_0(G)^+$. By Theorem \ref{thm:K0H0}, there is an ordered group isomorphism $K_0(C^*(G)) \rightarrow H_0(G))^+$ that maps $[1_V]_0$ to $[1_V]$. Since the equivalence classes $[1_V]$, for $V \subset G^{(0)}$ compact open, generate $H_0(G)^+$, then the corresponding equivalence classes $[1_V]_0$ generate $K_0(C^*(G))^+$.
\end{proof}

\section{Application to Deaconu-Renault groupoids}
\label{section:application}

Here, we apply the ordered group isomorphism from Theorem \ref{thm:HK} to characterise the AF embeddability of the C*-algebra of a Deaconu-Renault groupoid. First, we define this class of groupoid in Subsection \ref{section:preliminaries:DRgroupoid}, then we prove the commutative diagram \ref{eqn:diagram} in Section \ref{section:diagram}, which we used in the proof of Theorem \ref{thm:main} in Subsection \ref{section:main}.

\subsection{Deaconu-Renault groupoids}
\label{section:preliminaries:DRgroupoid}

Let $X$ be a locally compact, Hausdorff, second countable space, and $\sigma: X \rightarrow X$ is a surjective local homeomorphism. We define the set
\begin{align*}
\mathcal{G} = \lbrace (x,k,y) \in X \times \mathbb{Z} \times X : \exists m, n \in \mathbb{N}, k = m - n, \sigma^m(x) = \sigma^n(y) \rbrace,
\end{align*}
and we endow it with the range and source maps $r,s: \mathcal{G} \rightarrow \mathcal{G}$ with $r(x,k,y) = (x,0,x)$ and $s(x,k,y) = (y,0,y)$. Let $\mathcal{G}^{(2)} = \lbrace ((x, k_1, y), (y,k_2 ,z)): (x, k_1, y), (y, k_2, z) \in \mathcal{G} \rbrace
$ and we define the product $\mathcal{G}^{(2)} \rightarrow \mathcal{G}$ and inverse operations $\mathcal{G} \rightarrow \mathcal{G}$ by
\begin{align*}
(x, k_1, y)(y, k_2, z) = (x, k_1 + k_2, z),
\text{\hspace{10pt}}
(x,k,y)^{-1} = (y, -k, x).
\end{align*}
With these operations, $\mathcal{G}$ becomes a groupoid. The unit space is $\mathcal{G}^{(0)} = \lbrace (x,0,x) : x \in X \rbrace$, which is identified with $X$ by the map $(x,0,x) \mapsto x$. For the open subsets $A, B \subset X$ and for the natural numbers $m,n$ such that $\sigma\vert_A^m$ and $\sigma\vert_B^n$ are injective with $\sigma^m(A) = \sigma^n(B)$, define the subset
\begin{align*}
\mathcal{U}_{A,B}^{m,n}
= \lbrace (x, m - n, y) \in A \times \mathbb{Z} \times B : \sigma^m(x) = \sigma^n(y) \rbrace.
\end{align*}
Here we use the notation from \cite{ThomsenO2}. It follows from \cite[Theorem 1]{Deaconu} that the sets $\mathcal{U}_{A,B}^{m,n}$ generate a topology that makes $\mathcal{G}$ a topological groupoids which is locally compact, Hausdorff, second countable, \'etale. We denote the groupoid $\mathcal{G}$ endowed with this topology a \emph{Deaconu-Renault groupoid}.

The first example of this groupoid was described in Renault's book \cite[Section III.2]{Renault} by a class of groupoids whose C*-algebras are Cuntz algebras, where $X$ is the set of sequences $(x_i)_{i \in \mathbb{N}}$, with $x_i \in \lbrace 0, 1, \dots, n \rbrace$ (or $x_i \in \mathbb{N}$ for the Cuntz algebra $\mathcal{O}_\infty$), and $\sigma$ was given by $\sigma(x_0 x_1 x_2 \dots) = x_1 x_2 \dots$ Deaconu \cite{Deaconu} generalises Renault's definition by considering compact Hausdorff spaces $X$ and assumes that $\sigma:X \rightarrow X$ is a covering map. Under these assumptions, he proved that when $\sigma$ is a homeomorphism, then $C^*(\mathcal{G})$ is isomorphic to the crossed product $C(X) \rtimes_\sigma \mathbb{Z}$. Anantharaman-Delaroche \cite[Example 1.2 (c)]{Anantharaman-Delaroche} generalised it to the case where $\sigma$ is a surjective local homeomorphism, and $X$ is locally compact. She also gave a sufficient conditions to make the C*-algebra of this groupoid purely infinite \cite[Propositions 4.2 and 4.3]{Anantharaman-Delaroche}. Then Renault \cite{Renault-Cuntzlike} generalised this definition by considering a map $\sigma$ that is defined only on an open subset of $X$. For the purpose of our paper, here we consider the definition by Anantharaman-Delaroche, and we impose the additional conditions on the space $X$ to be Hausdorff and second countable.

By \cite[Proposition 2.9]{Renault-Cuntzlike}, the groupoid is also amenable. Note that there are different notions of amenability for groupoids: topological amenable and measurewise amenable. By Theorem 3.3.7 of \cite{ADRenault}, these two notions coincide for groupoids that are locally compact, Hausdorff and \'etale. In this case, \cite[Proposition 6.1.8]{ADRenault} implies that the reduced and full C*-algebras of the groupoid coincide.

\begin{remark}
    We use the notation $c^{-1}(0)$ for the groupoid $\bigcup_{n=1}^\infty R(\sigma^n)$ in Example \ref{example:cinv0} because of the following. If $\mathcal{G}$ is the Deaconu-Renault groupoid and $c: \mathcal{G} \rightarrow \mathbb{Z}$ is given by $c(x,k,y) = k$ (in this case, $c$ is called the \emph{canonical cocycle}), then $c^{-1}(0)$ is isomorphic to the subgroupoid $\lbrace (x,0,y) \in \mathcal{G} \rbrace$ of $\mathcal{G}$. The isomorphism is defined by $(x,y) \mapsto (x,0,y)$.
\end{remark}

\subsection{Skew-products and crossed products}
\label{section:skewproducts}

As explained in the introduction, we use an isomorphism between the C*-algebras a skew-product groupoid and the corresponding crossed product to understand when the C*-algebra of a Deaconu-Renault groupoid $\mathcal{G}$ is AFE. This isomorphism allows us to apply the results of Brown's Theorem (Theorem \ref{thm:brown}). In this section, we use Renault's book \cite{Renault} as a reference.

\begin{definition}
\label{def:Gc} \cite[Definition I.1.6]{Renault}
Let $G$ be a groupoid and $c: G \rightarrow \mathbb{Z}$ a homomorphism. Define the sets
\begin{align*}
G(c) = \lbrace (g,a) : g \in G, a \in \mathbb{Z} \rbrace
\hspace{10pt}\text{and}\hspace{10pt}
G(c)^{(0)} = G^{(0)} \times \mathbb{Z}.
\end{align*}
Define the maps $r, s: G(c) \rightarrow G^{(0)} \times \mathbb{Z}$ by $r(g,a) = (r(g), a)$ and $s(g,a) = (s(g), a + c(g))$. Consider the set
\begin{align}
\label{eqn:Gc2}
G(c)^{(2)} = \lbrace ((g,a), (h,a + c(g))): (g,h) \in G^{(2)}, a \in \mathbb{Z} \rbrace,
\end{align}
and define the product $G(c)^{(2)} \rightarrow G(c)$ and the inverse $G(c) \rightarrow G(c)$ by
\begin{align}
(g,a)(h, a + c(g)) = (gh,a)
\hspace{15pt}\text{and}\hspace{15pt}
(g,a)^{-1} = (g^{-1},a+c(g))
\label{eqn:Gcprod},
\end{align}
respectively. Then $G(c)$ becomes a groupoid, called a \emph{skew-product}.
\end{definition}

If $c$ is a continuous cocycle (i.e., if $c: G \rightarrow \mathbb{Z}$ is a continuous homomorphism) and we equip $G(c)$ with the induced topology from $G \times \mathbb{Z}$, then $G(c)$ inherits the usual topogical properties of $G$. Recall that we assume that $G$ is locally compact, Hausdorff, second countable and \'etale.

\begin{lemma}
\label{lemma:Gctopology}
Equip $G(c)$ with the topology induced from $G \times \mathbb{Z}$. Then $G(c)$ is also a locally compact, Hausdorff, second countable, \'etale groupoid. Moreover, if $G$ is totally disconnected, then $G(c)$ is totally disconnected, too. If $G$ is measurewise amenable, then $G(c)$ is also measurewise amenable.
\end{lemma}

Note that can identify the unity space $G(c)^{(0)}$ with $\mathcal{X} \times \mathbb{Z}$. For $f, f_1, f_2 \in C_c(G(c))$ and $(g, a) \in G \times \mathbb{Z}$, the operations on $C_c(G(c))$ are given by
\begin{align*}
(f_1 \cdot f_2)(g,a)
= \sum_{g_1 g_2 = g} f_1(g_1,a) f_2(g_2, a + c(g_1)),
\text{ and }
f^*(g,a) = \overline{f(g^{-1}, a + c(g))}.
\end{align*}

Assume that $G$ is a locally compact, Hausdorff, second countable, \'etale, measurewise amenable groupoid, and suppose that $c: G \rightarrow \mathbb{Z}$ is a continuous cocycle. We define a dynamical system $\alpha^c$ induced by $c$. For every $z \in \mathbb{T}$, $f \in C_c(G)$ and $g \in G$, we set $\alpha_z^c(f)(g) = z^{c(g)} f(g)$. By Proposition II.5.1 of Renault's book \cite{Renault}, there exists a unique dynamical system defined by the equation above. We make an abuse of notation by denoting this dynamical system $\alpha^c: \mathbb{T} \rightarrow \mathrm{Aut}(C^*(G))$. The following proposition by Renault \cite[Proposition II.5.7]{Renault} gives an isomorphism between the C*-algebra of skew-product groupoids and the corresponding crossed products. For an introduction to crossed products, see Chapters 1 and 2 of Dana William's book \cite{Williams-crossedproducts}.

\begin{proposition}
\label{prop:crossedproduct-isomorphism}
There exists a unique isomorphism $\rho \rtimes u: C^*(G) \rtimes_{\alpha^c} \mathbb{T} \rightarrow C^*(G(c))$ such that
\begin{align*}
\rho \rtimes u(f)(g,a)
&= \int_\mathbb{T} z^{-c(g) - a} f(z)(g) dz
\end{align*}
for $f \in C_c(\mathbb{T}, C_c(G))$ and $(g,a) \in G(c)$. Moreover, the inverse of $\rho \rtimes u$ is given by
\begin{align*}
(\rho \rtimes u)^{-1}(F)(z)(g)
&= \frac{1}{2\pi} \sum_{a \in \mathbb{Z}} z^{c(g) + a} F(g,a),
\end{align*}
for $F \in C_c(G(c))$, $z \in \mathbb{T}$, $g \in G$.
\end{proposition}

Now we find a formula for the action $\beta$ on $C^*(G(c))$ induced by the dual action $\widehat{\alpha}^c$. We will apply this formula in Chapter \ref{section:main} for Deaconu-Renault groupoids in order to understand in more detail the condition $H_\beta \cap K_0(B)^+ = \lbrace 0 \rbrace$, given by Brown's theorem.

\begin{proposition}
\label{prop:beta}
Let $\rho \rtimes u: C^*(G) \rtimes_{\alpha^c} \mathbb{T} \rightarrow C^*(G(c))$ be the isomorphism of Proposition \ref{prop:crossedproduct-isomorphism}. Then there exists a unique dynamical system $\beta: \mathbb{Z} \rightarrow \mathrm{Aut}(C^*(G(c)))$ such that
\begin{align*}
\beta(F)(g,a) = F(g, a + 1)
\hspace{15pt}
\text{for $F \in C_c(G(c))$, $(g,a) \in G(c)$,}
\end{align*}
and $\beta = (\rho \rtimes u) \circ \widehat{\alpha}^c \circ (\rho \rtimes u)^{-1}$ for all $n \in \mathbb{Z}$. Moreover, $C^*(G(c)) \rtimes_\beta \mathbb{Z}$ and $C^*(G) \rtimes_{\alpha^c} \mathbb{T} \rtimes_{\widehat{\alpha}^c} \mathbb{Z}$ are isomorphic.
\end{proposition}
\begin{proof}
    Let $\widehat{\alpha}^c: \mathbb{Z} \rightarrow \mathrm{Aut}(C^*(G)  \rtimes_{\alpha^c} \mathbb{T})$ be the dual action, which is given by $\widehat{\alpha}^c_n(f)(z) = z^{-n} f(z)$, for $f \in C(\mathbb{T}, C^*(G))$, $z \in \mathbb{T}$, $n \in \mathbb{Z}$. For each $n \in \mathbb{Z}$, let $\beta_n = (\rho \rtimes u)^{-1} \circ \widehat{\alpha}^c_n \circ \rho \rtimes u$. Since $\rho \rtimes u$ is an isomorphism and $\widehat{\alpha}^c$ is an automorphism, we have that $\beta: \mathbb{Z} \rightarrow \mathrm{Aut}(C^*(G(c)))$ is a dynamical system.

    Fix $n \in \mathbb{Z}$. By continuity of $\beta_n$, finding the formula of this automorphism on $C_c(G(c))$ is sufficient to determine $\beta_n$ uniquely. Let $F \in C_c(G(c))$ and $(g,a) \in G(c)$. Then, Proposition \ref{prop:crossedproduct-isomorphism},
    \begin{align*}
        \beta_n(F)(g,a)
        &= (\rho \rtimes u) \circ \widehat{\alpha^c}_n \circ (\rho \rtimes u)^{-1}(F)(g,a)\\
        &= \frac{1}{2\pi}\int_\mathbb{T} z^{-c(g) - a} \widehat{\alpha^c}_n \circ (\rho \rtimes u)^{-1}(F)(z)(g)dz\\
        &= \frac{1}{2\pi}\int_\mathbb{T} z^{-c(g) -a} z^{-n} (\rho \rtimes u)^{-1}(F)(z)(g) dz \\
        &= \frac{1}{2\pi}\int_\mathbb{T} z^{-c(g)-a-n} \sum_{b \in \mathbb{Z}} z^{c(g) + b} F(g,b) dz \\
        &= \frac{1}{2\pi}\sum_{b \in \mathbb{Z}} F(g,b) \int_\mathbb{T} z^{-a-n+b} dz\\
        &= F(g,a+n).
    \end{align*}
    By definition, we have that $\rho \rtimes u \circ \widehat{\alpha}^c = \beta \circ \rho \rtimes u$. Then $\rho \rtimes u$ is an equivariant isomorphism from $(C^*(G) \rtimes_{\alpha^c} \mathbb{T}, \widehat{\alpha}^c, \mathbb{Z})$ onto $(C^*(G(c)), \beta, \mathbb{Z})$. By \cite[Lemma 2.65]{Williams-crossedproducts}, $C^*(G(c)) \rtimes_\beta \mathbb{Z}$ and $C^*(G) \rtimes_{\alpha^c} \mathbb{T} \rtimes_{\widehat{\alpha}^c} \mathbb{Z}$ are isomorphic.
\end{proof}

\subsection{The Commutative Diagram}
\label{section:diagram}

Recall from the introduction that we want to find a group which is isomorphic to $K_0(B)$ and more manageable, where $\mathcal{G}$ is a Deaconu-Renault groupoid,  $B$ is the C*-algebra of the skew product groupoid $\mathcal{G}(c)$, and $c: \mathcal{G} \rightarrow \mathbb{Z}$ is the continuous cocycle given by $c(x,k,y) = k$, called the \emph{canonical cocycle}. By studying this new group which is isomorhic to $K_0(B)$, we will find a condition on $\sigma$ that is equivalent to the condition $H_\beta \cap K_0(B)^+ = \lbrace 0 \rbrace$ from Theorem \ref{thm:brown}.

We prove the diagram below by using the techniques for homology groups described in Section \ref{section:homology}, and by applying isomorphism from Theorem \ref{thm:K0H0}.
\begin{equation}
\tag{\ref{eqn:diagram}}
\begin{tikzcd}
      K_0(B) \arrow[r, "K_0(\beta)"] \arrow{d}{\cong}[swap]{\circledRed{1} \hspace{6pt}}
      & K_0(B) \arrow{d}{}[swap]{\cong} \\
      H_0(\mathcal{G}(c))  \arrow[r, "{[\widetilde{\beta}]}_{\mathcal{G}(c)}"] \arrow{d}{\cong}[swap]{\circledRed{2} \hspace{6pt}} & H_0(\mathcal{G}(c)) \arrow{d}{}[swap]{\cong}\\
      H_0(c^{-1}(0))  \arrow{r}{[\sigma_\ast] }[swap]{} & H_0(c^{-1}(0))   
\end{tikzcd}
\end{equation}
The action $\beta$ is described in Proposition \ref{prop:beta}. For a Deaconu-Renault groupoid $\mathcal{G}$, the action $\beta \in \mathrm{Aut}(B)$ is such that $\beta(f)(x,k,y,a) = f(x,k,y, a+ 1)$ for $(x,k,y,a) \in \mathcal{G}(c)$. We prove the first part of the diagram using Theorem \ref{thm:K0H0}, and we prove the second part by applying results on homology groups for groupoids.

To avoid confusion among the equivalence classes for different groups, we assume that $[\hphantom{f}]_{K_0(B)}$, $[\hphantom{f}]_{\mathcal{G}(c)}$, $[\hphantom{f}]$ denote the equivalence classes in $K_0(B), H_0(\mathcal{G}(c)), H_0(c^{-1}(0))$, respectively. We fix $X$ to be a locally compact, Hausdorff, second countable, totally disconnected space, and we let $\sigma: X \rightarrow X$ be a surjective local homeomorphism. Here $\mathcal{G}$ denotes the Deaconu-Renault groupoid for $\sigma: X \rightarrow X$, and we fix the notation $B = C^*(\mathcal{G}(c))$.

\subsubsection{Part 1 of the diagram}

First we define the map $\widetilde{\beta}$ of the diagram. This is a map $\widetilde{\beta}: C_c(X \times \mathbb{Z}, \mathbb{Z}) \rightarrow C_c(X \times \mathbb{Z}, \mathbb{Z})$ given by $\widetilde{\beta}(f)(x,a) = f(x, a+1)$, for $f \in C_c(X \times \mathbb{Z}, \mathbb{Z})$ and $(x,a) \in X \times \mathbb{Z}$. Note that $\widetilde{\beta}$ is the restriction of $\beta$ to $C_c(X \times \mathbb{Z}, \mathbb{Z})$. Now we study the ordered group homomorphism induced by $\widetilde{\beta}$.

\begin{lemma}
\label{lemma:betatildeH0}
The map $[\widetilde{\beta}]_{\mathcal{G}(c)}: H_0(\mathcal{G}(c)) \rightarrow H_0(\mathcal{G}(c))$ is well-defined and is an ordered group homomorphism.
\end{lemma}
\begin{proof}
    Recall from Remark \ref{remark:inducedhomomorphismH0} that $[\widetilde{\beta}]_{\mathcal{G}(c)}$ if given by $[\widetilde{\beta}]_{\mathcal{G}(c)}([f]_{\mathcal{G}(c)}) = [\widetilde{\beta}(f)]_{\mathcal{G}(c)}$. Let $f_1, f_2 \in C_c(X \times \mathbb{Z}, \mathbb{Z})$ be equivalent in $H_0(\mathcal{G}(c))$. Then there is an $F \in C_c(\mathcal{G}(c), \mathbb{Z})$ such that $f_1 = f_2 + \partial_1 F$. Given $(x,a) \in X \times \mathbb{Z}$, we have
    \begin{align*}
        \beta(\partial_1 F)(x,a)
        &= \partial_1 F(x, a+ 1) \\
        &= s_\ast(F)(x, a+1) - r_\ast(F)(x, a+1) \\
        &= \sum_{(g,b):s(g,b) = (x, a +1)} F(g,b) - \sum_{(g,b):r(g,b) = (x, a +1)} F(g,b) \\
        &= \sum_{g \in \mathcal{G}: s(g) = x} F(g,a+1 - c(g)) - \sum_{g \in \mathcal{G}: r(g) = x} F(g,a+1) \\
        &= \sum_{g \in \mathcal{G}: s(g) = x} \beta(F)(g, a - c(g)) - \sum_{g \in \mathcal{G}: r(g) = x} \beta(F)(g,a) \\
        &= s_\ast(\beta(F))(x,a) - r_\ast(\beta(F))(x,a) \\
        &= \partial_1(\beta(F))(x,a).
    \end{align*}
    Thus, $\widetilde{\beta}(f_1) = \widetilde{\beta}(f_2) + \partial_1(\beta(F))$. Note that $\beta(F) \in C_c(\mathcal{G}(c), \mathbb{Z})$. Then $\beta(f_1)$ and $\beta(f_2)$ are equivalent and therefore $[\widetilde{\beta}]_{\mathcal{G}(c)}$ is well-defined. It straightforward from the definition that $[\widetilde{\beta}]_{\mathcal{G}(c)}$ is an ordered group homomorphism.
\end{proof}

\begin{proposition}
The first part of the diagram \eqref{eqn:diagram} is commutative.
\end{proposition}
\begin{proof}
    Since $H_0(\mathcal{G}(c))$ is generated by elements of the form $[1_{V \times \lbrace a \rbrace}]_{\mathcal{G}(c)}$ for $V \subset X$ compact open and $a \in \mathbb{Z}$, then we only need to prove that, for all $V \subset X$ compact open and $a \in \mathbb{Z}$,
    \begin{align*}
        K_0(\beta) \circ \mu([1_{V \times \lbrace a \rbrace}]_{\mathcal{G}(c)})
        = \mu \circ [\widetilde{\beta}]_{\mathcal{G}(c)}([1_{V \times \lbrace a \rbrace}]_{\mathcal{G}(c)}).
    \end{align*}
    So, fix $V$ and $a$. Then
    \begin{align*}
        K_0(\beta) \circ \mu([1_{V \times \lbrace a \rbrace}]_{\mathcal{G}(c)})
        = K_0(\beta)([1_{V \times \lbrace a \rbrace}]_{K_0(B)})
        = [\beta(1_{V \times \lbrace a \rbrace})]_{K_0(B)}
        %= [ \widetilde{\beta}(1_{V \times \lbrace a \rbrace})]_{K_0(B)}
        %= \mu([\widetilde{\beta}(1_{V \times \lbrace a \rbrace})]_{\mathcal{G}(c)})
        = \mu \circ [\widetilde{\beta}]_{\mathcal{G}(c)}([1_{V \times \lbrace a \rbrace}]_{\mathcal{G}(c)}).
    \end{align*}
Therefore the result holds.
\end{proof}

\subsubsection{Part 2 of the diagram}

In the second part of the diagram, we have an isomorphism between homology groups of different groupoids. We will obtain this isomorphism by applying the techniques of homological similarity of groupoids. Until the rest of this section, we fix $\varphi: X \rightarrow X$ be a continuous section of the surjective local homeomorphism $\sigma: X \rightarrow X$ that defines $\mathcal{G}$. Let us define the map that induces the similarity of the groupoids $\mathcal{G}(c)$ and $c^{-1}(0)$.

\begin{lemma}
\label{lemma:rhophi}
There exists a homomorphism $\rho_\varphi: \mathcal{G}(c) \rightarrow c^{-1}(0)$ given by
\begin{align*}
\rho_\varphi(x,k,y,a) = (\varphi^a(x), \varphi^{a+k}(y)).
\end{align*}
Moreover, $\rho_\varphi$ is a local homeomorphism.
\end{lemma}
\begin{proof}
    First we show that $\rho_\varphi$ is well-defined. Let $(x,k,y,a) \in \mathcal{G}(c)$, and let $m,n \in \mathbb{N}$ be such that $\sigma^m(x) = \sigma^n(y)$ and $k = m - n$. We divide the proof into four cases, considering the signs of $a$ and $a+k$. Suppose $a \geq 0$ and $a + k \geq 0$. Then $\sigma^a(\varphi^a(x)) = x$ and $\sigma^{a+k}(\varphi^{a+k}(y)) = y$. This implies that
    \begin{align*}
        \sigma^{m+a}(\varphi^a(x))
        &= \sigma^m(\sigma^a(\varphi^a(x)))
        = \sigma^m(x)
        = \sigma^n(y), \text{\hspace{10pt}and} \\
        \sigma^{m+a}(\varphi^{a+k}(y))
        &= \sigma^{n+a+k}(\varphi^{a+k}(y)) %\text{\hspace{10pt} since $m=n+k$,}\\
        =\sigma^n(\sigma^{a+k}(\varphi^{a+k}(y)))
        = \sigma^n(y).
    \end{align*}
    Hence, $(\varphi^a(x), \varphi^{a+k}(y)) \in c^{-1}(0)$. The proof for the remaining cases is analogous. Then $\rho_\varphi$ is well-defined.

    Next we show that $\rho_\varphi$ is a homomorphism. Given $(x,k,y,a), (y,l,z,a+k) \in \mathcal{G}(c)$, we have
    \begin{align*}
        \rho_\varphi((x,k,y,a)(y,l,z,a+k))
        &= \rho_\varphi(x,k+l,z,a) \\
        &= (\varphi^a(x), \varphi^{a+k+l}(z)) \\
        &= (\varphi^a(x), \varphi^{a+k}(y))(\varphi^{a+k}(y), \varphi^{a+k+l}(z)) \\
        &= \rho_\varphi(x,k,y,a)\rho_\varphi(y,l,z,a+k).
    \end{align*}
    Then $\rho_\varphi$ is a homomorphism.

    Finally, we prove that $\rho_\varphi$ is a local homeomorphism. Fix $a \in \mathbb{Z}$. Then the restriction $\rho_\varphi^{(0)}\vert_{X \times \lbrace a \rbrace}: X \times \lbrace a \rbrace \rightarrow X$ given by $(x,a) \mapsto \varphi^a(x)$ is a local homeomorphism because $\varphi^a$ is a local homeomorphism. By Lemma \ref{lemma:rhon-equivalent}, $\rho_\varphi$ is a local homeomorphism.
\end{proof}

\begin{proposition}
\label{prop:similarityGcc0}
$H_0(\rho_\varphi): H_0(\mathcal{G}(c)) \rightarrow H_0(c^{-1}(0))$ is an isomorphism and has inverse $H_0(\eta)$, where $\eta: c^{-1}(0) \rightarrow \mathcal{G}(c)$ is defined by $\eta(x,y) = (x,0,y,0)$.
\end{proposition}
\begin{proof}
    It follows from the definition that $\eta$ is a well-defined homomorphism. Note that $\eta^{(0)}$ is defined by $\eta^{(0)}(x) = (x,0)$. Thus $\eta^{(0)}$ is a local homeomorphism. By Lemma \ref{lemma:rhon-equivalent}, $\eta$ is also a local homeomorphism. It is straightforward that $\rho_\varphi \circ \eta = \mathrm{id}_{c^{-1}(0)}$.

    Now we show that $\eta \circ \rho_\varphi$ is homologically similar to $\mathrm{id}_{\mathcal{G}(c)}$. Let $\theta: X \times \mathbb{Z} \rightarrow \mathcal{G}(c)$ be defined by $\theta(x,a) =  (x,-a, \varphi^a(x), a)$. Note that $\theta$ is well-defined. In fact, if $a \geq 0$, we have that $\sigma^a(\varphi^a(x)) = x$. Then $(x,-a,\varphi^a(x),a) \in \mathcal{G}(c)$. If $a < 0$, then $\varphi^a(x) = \sigma^{-a}(x)$, which implies that $(x,-a, \varphi^a(x),a) \in \mathcal{G}(c)$. This function is also continuous because it is continuous on each coordinate.

    Let $g = (x,k,y,a) \in \mathcal{G}(c)$. Then
    \begin{align*}
        \theta(r(g)) \eta \circ \rho_\varphi(g)
        &= \theta(x,a) \eta \circ \rho_\varphi(x,k,y,a) \\
        &= (x,-a, \varphi^a(x), a) \eta \circ \rho_\varphi(x,k,y,a) \\
        &= (x, - a, \varphi^a(x), a) \eta(\varphi^a(x), \varphi^{a+k}(y)) \\
        &= (x, - a, \varphi^a(x), a) (\varphi^a(x), 0, \varphi^{a+k}(y),0) \\
        &= (x, - a, \varphi^{a+k}(y), a),
        \text{\hspace{7pt} and}\\
        \mathrm{id}_{\mathcal{G}(c)}(g) \theta(s(g))
        &= (x,k,y,a)\theta(y,a+k) \\
        &= (x,k,y,a)(y,-a -k, \varphi^{a+k}(y), a+k)\\
        &= (x, -a, \varphi^{a+k}(y),a).
    \end{align*}
    So $\eta \circ \rho_\varphi$ and $\mathrm{id}_{\mathcal{G}(c)}$ are homologically similar. Therefore, $\mathcal{G}(c)$ and $c^{-1}(0)$ are homologically similar. Moreover, by Corollary \ref{corollary:isoHn}, we have an isomorphism $H_0(\rho_\eta): H_0(\mathcal{G}(c)) \rightarrow H_0(c^{-1}(0))$  with inverse $H_0(\eta)$.
\end{proof}

\begin{proposition}
The second part of the diagram \eqref{eqn:diagram} is commutative.
\end{proposition}
\begin{proof}
    By Proposition \ref{prop:similarityGcc0}, $H_0(\rho_\varphi): H_0(\mathcal{G}(c)) \rightarrow H_0(c^{-1}(0))$ is an isomorphism with inverse $H_0(\eta)$. In order to show that the diagram above commutes, we prove that $(\rho_\varphi^{(0)})_\ast \circ \widetilde{\beta} \circ \eta^{(0)}_\ast = \sigma_\ast$. Let $f \in C_c(X, \mathbb{Z})$ and $x \in X$. Note that the restriction $\rho_\varphi^{(0)}$ satisfies $\rho_\varphi^{(0)}(y,a) = \varphi^a(y)$. Then
    \begin{align*}
        (\rho_\varphi)_\ast^{(0)} \circ \widetilde{\beta} \circ \eta_\ast^{(0)}(f)(x)
        &= \sum_{(y,a): \rho_\varphi(y,a) = x} \widetilde{\beta} \circ \eta_\ast^{(0)}(f)(y,a)\\
        &= \sum_{a \in \mathbb{Z}} \sum_{y: \varphi^a(y) = x} \widetilde{\beta} \circ \eta_\ast^{(0)}(f)(y,a) \\
        &= \sum_{a \in \mathbb{Z}} \sum_{y: \varphi^a(y) = x} \eta_\ast^{(0)}(f)(y,a + 1).
    \end{align*}
    
    Note that $\eta_\ast^{(0)}(f)(y,a+1) = 1$ if $a = -1$, and zero otherwise. Then
    \begin{align*}
        (\rho_\varphi)_\ast^{(0)} \circ \widetilde{\beta} \circ \eta_\ast^{(0)}(f)(x)
        &= \sum_{y: \varphi^{-1}(y) = x} f(y)
        = \sum_{y: \sigma(y) = x} f(y)
        = \sigma_\ast(f)(x).
    \end{align*}
    This implies that $[\sigma_\ast] = H_0(\rho) [\widetilde{\beta}] H_0(\eta)$. Therefore, the diagram commutes.
\end{proof}

%\section{Main Result}
\subsection{AF embeddability for Deaconu-Renault groupoids}
\label{section:main}

Here we show Theorem \ref{thm:main}, which gives a condition on the map $\sigma:X \rightarrow X$, characterising when the corresponding groupoid C*-algebra $C^*(\mathcal{G})$ is AF embeddable. We prove this result by applying the commutative diagram of Section \ref{section:diagram} in combination with Theorem \ref{thm:brown}.

We begin this section by defining $\sigma_\ast(F)$ for functions $F \in C_c(c^{-1}(0), \mathbb{Z})$, using the fact that $c^{-1}(0)$ is the increasing union of the sets $R(\sigma^n)$.

\begin{lemma}
\label{lemma:sigmaastRsigma}
Fix $n \in \mathbb{N}$ and let $F \in C_c(R(\sigma^n), \mathbb{Z})$. We define $\sigma_\ast(F) \in C_c(R(\sigma^n), \mathbb{Z})$ by
\begin{align}
\label{eqn:sigmaastRsigma}
 \sigma_\ast(F)(x,y) = \sum_{\substack{u: \hphantom{f} \sigma(u) = x \\ v: \hphantom{f} \sigma(v) = y}} F(u,v).
\end{align}
Then $\sigma_*(F)$ is well-defined. In particular, if $n \geq 1$, we have $\sigma_\ast(F) \in C_c(R(\sigma^{n-1}), \mathbb{Z})$.
\end{lemma}
\begin{proof}
    Define the sets
    \begin{align*}
        \mathcal{U}_{A,B}^n
        = \lbrace (x,y) : x \in A, y \in B, \sigma^n(A) = \sigma^n(B) \rbrace,
    \end{align*}
    where $A, B$ are compact open wiht $\sigma^n(A) = \sigma^n(B)$, and $\sigma^n$ is injective on $A$ and $B$. Then one can show that $\sigma(1_{\mathcal{U}_{A, B}^{n, n}}) = 1_{\mathcal{U}_{\sigma(A), \sigma(B)}^{n-1, n-1}}$. The sets $\mathcal{U}_{A,B}^n$ generate the topology of $R(\sigma^n)$, then the functions $1_{\mathcal{U}_{A, B}^{n, n}}$. By linearity, $\sigma_\ast(F)$ is well-defined.
\end{proof}
\begin{customthm}{1}
%\begin{theorem}
%\label{thm:main}
    Let $\sigma$ be a surjective local homeomorphism on a locally compact, Hausdorff, second countable, totally disconnected space $X$. Denote by $\mathcal{G}$ the Deaconu-Renault groupoid corresponding to $\sigma$. Then the following are equivalent:
    \begin{enumerate}[(i)]
    \item $C^*(\mathcal{G})$ is AF embeddable,
    \item $C^*(\mathcal{G})$ is quasidiagonal,
    \item $C^*(\mathcal{G})$ is stably finite,
    \item $\mathrm{Im}(\sigma_\ast - \mathrm{id}) \cap C_c(X,\mathbb{N}) = \lbrace 0 \rbrace,$
    \end{enumerate}
    where the map $\sigma_\ast: C_c(X, \mathbb{Z}) \rightarrow C_c(X,\mathbb{Z})$ is defined by
    \begin{align*}
        \sigma_\ast(f)(x) = \sum_{y:\sigma(y) = x} f(y)
        \hspace{10pt}
        \text{for }x \in X, f \in C_c(X, \mathbb{Z}).
    \end{align*}
%\end{theorem}
\end{customthm}
\begin{proof}
    Let $c: \mathcal{G} \rightarrow \mathbb{Z}$ be the continuous cocycle given by $c(x,k,y) = k$ for $(x,k,y) \in \mathcal{G}$. By \cite[Corollary 5.2]{FKPS}, we have that $\mathcal{G}(c)$ is an AF groupoid, and therefore the C*-algebra $C^*(G(c))$ is AF. For simplicity of notation, we denote this C*-algebra by $B$.

    Let $\beta$ be the action on $B$ corresponding to the dual action $\hat{\alpha}^c$ and induced by the isomorphism $B \cong C^*(G) \rtimes_{\alpha^c} \mathbb{T}$ from Proposition \ref{prop:beta}. Then Brown's theorem (Theorem \ref{thm:brown}) implies the equivalent conditions bellow.

    \begin{enumerate}[(i')]
    \item $B \rtimes_\beta \mathbb{Z}$ is AF embeddable,
    \item $B \rtimes_\beta \mathbb{Z}$ is quasidiagonal,
    \item $B \rtimes_\beta \mathbb{Z}$ is stably finite,
    \item $H_\beta \cap K_0(B)^+ = \lbrace 0 \rbrace$.
    \end{enumerate}

    By Proposition \ref{prop:beta}, $B \rtimes_\beta \mathbb{Z}$ and $C^*(\mathcal{G}) \rtimes_{\alpha^c} \mathbb{T} \rtimes_{\widehat{\alpha}^c} \mathbb{Z}$ are isomorphic. It follows from Takai duality that $C^*(\mathcal{G}) \rtimes_{\alpha^c} \mathbb{T} \rtimes_{\widehat{\alpha}^c} \mathbb{Z}$ and $C^*(\mathcal{G})$ are stably isomorphic. So, $B \rtimes_\beta \mathbb{Z}$ and $C^*(\mathcal{G})$ are stably isomorphic. Since AF embeddability, quasidiagonality and stable finiteness are preserved under stable isomorphisms, then we have the equivalence below:

    \begin{enumerate}[(i)]
    \item $C^*(\mathcal{G})$ is AF embeddable,
    \item $C^*(\mathcal{G})$ is quasidiagonal,
    \item $C^*(\mathcal{G})$ is stably finite,
    \item[(iv')] $H_\beta \cap K_0(B)^+ = \lbrace 0 \rbrace$.
    \end{enumerate}

    Now we study condition (iv'). By the commutative diagram \eqref{eqn:diagram}, (iv') is false if, and only if, there are non-zero $f \in C_c(X, \mathbb{Z})$ and $h \in C_c(X, \mathbb{N})$ such that
    \begin{align*}
    [\sigma_\ast(f) - f] = [h].
    \end{align*}

    Suppose the equation above holds. Then there exists $F \in C_c(c^{-1}(0), \mathbb{Z})$ such that $\sigma_\ast(f) - f = h + \partial_1 F$. Then there exists $n \geq 1$ such that $F \in C_c(R(\sigma^n), \mathbb{Z})$. Define $\partial_1^{(n)}: C_C(R(\sigma^n), \mathbb{Z}) \rightarrow C_C(R(\sigma^n), \mathbb{Z})$ by $\partial_1^{(n)} = (s\vert_{R(\sigma^n)})_\ast - (r\vert_{R(\sigma^n)})_\ast$. It is straightforward that $\partial_1 F = \partial_1^{(n)} F$. Hence,
    \begin{align}
    \label{eqn:sigmafhpartialnF}
        \sigma_\ast(f) - f = h + \partial_1^{(n)} F.
    \end{align}
    Moreover,
    \begin{align*}
        \sigma_\ast(\partial^{(n)}_1(F))(x)
        %&= \sum_{y: \sigma(y) = x} \partial_1^{(n)}(F)(y) \\
        &= \sum_{\substack{y: \sigma(y) = x \\ z: \sigma^n(z) = \sigma^n(y)}} [F(z,y) - F(y,z)] \\
        %&= \sum_{y: \sigma(y) = x} \left( \sum_{z: \sigma^n(z) = \sigma^{n-1}(x)} [F(z,y) - F(y,z)] \right), \\ %\hspace{10pt}\text{because $n \geq 1$,} \\
        %&= \sum_{y: \sigma(y) = x} \left( \sum_{v: \sigma^{n-1}(v) = \sigma^{n-1}(x)} \left( \sum_{z: \sigma(z) = v} [F(z,y) - F(y,z)] \right) \right) \\
        &= \sum_{\substack{v: \sigma^{n-1}(v) = \sigma^{n-1}(x) \\ y: \sigma(y) = x \\ z:\sigma(z) = v}} [F(z,y) - F(y,z)] \\
        &= \sum_{v: \sigma^{n-1}(v) = \sigma^{n-1}(x)} [\sigma_\ast(F)(v,x) - \sigma_\ast(F)(x,v)] \\
        &= \partial_1^{(n-1)}(\sigma_\ast(F))(x)
    \end{align*}

    It follows from Lemma \ref{lemma:sigmaastRsigma} that $\sigma_\ast(F) \in C_c(R(\sigma^n))$. By applying $\sigma^n_\ast$ on both sides of the equation \eqref{eqn:sigmafhpartialnF} and setting $\widetilde{f} = \sigma_\ast^n(f)$ and $\widehat{h} = \sigma_\ast^n(h)$, we obtain
    \begin{align*}
    \sigma_\ast(\widetilde{f}) - \widetilde{f}
    = \widetilde{h} + \sigma_\ast^n(\partial_1^{(n)} F)
    = \widetilde{h}.
    \end{align*}

    We have just proved that $\mathrm{Im}(\sigma_\ast - \mathrm{id}) \cap C_c(X, \mathbb{N}) \neq \lbrace 0 \rbrace$. Now we prove the converse. Suppose there are $f \in C_c(X, \mathbb{N})$ and $h \in C_c(X, \mathbb{N})$ such that $h \neq 0$ and $\sigma_\ast(f) - f = h$. It follows from Lemma \ref{lemma:traceH0Rsigma} that $[h] \neq 0$. Then equation (iv') holds, which implies that $C^*(\mathcal{G})$  is not AF embeddable.
\end{proof}

We will use the following corollary to show that our main theorem generalises a known theorem for graph algebras.

\begin{corollary}
\label{corollary:ofthm1}
    The following are equivalent:
    \begin{enumerate}
        \item $C^*(\mathcal{G})$ is AF embeddable,
        \item $\mathrm{Im}(\sigma_\ast - \mathrm{id}) \cap C_c(X, \mathbb{N}) = \lbrace 0 \rbrace,$
        \item $\mathrm{Im}(\sigma_\ast^n - \mathrm{id}) \cap C_c(X, \mathbb{N}) = \lbrace 0 \rbrace$ for some $n \geq 1$,
        \item $\mathrm{Im}(\sigma_\ast - \mathrm{id}) \cap C_c(X, \mathbb{N}) = \lbrace 0 \rbrace$ for all $n \geq 1$.
    \end{enumerate}
\end{corollary}
\begin{proof}
    The equivalence $(1) \Leftrightarrow (2)$ follows from Theorem \ref{thm:main}, and the implications $(4) \Rightarrow (2) \Rightarrow (3)$ are straightforward. We prove $(3) \Rightarrow (4)$ by contrapositive. Suppose $(4)$ is false, and fix $n$. Then, for some $m \geq 1$, there are nonzero $f \in C_c(X, \mathbb{Z}$ and $h \in C_c(X, \mathbb{N})$ such that $\sigma_\ast(f) - f = h$. Then, for all $k$, we have
    \begin{align*}
        \sigma^k_\ast(\sigma_\ast^m(f)) - \sigma_\ast^k(f) = \sigma_\ast^k(h).
    \end{align*}
    Let $F = \sigma_\ast^m(f)$ and $H = h + \sigma_\ast(h) + \dots + \sigma_\ast^{n-1}(h)$. Note that $H$ is non-negative and nonzero. By a telescoping sum, we have
    \begin{align*}
    \sigma^n_\ast(F) - F
    = \sum_{k=0}^{n-1} [\sigma_\ast^{k+1}(\sigma_\ast^m(f)) - \sigma_\ast^{k}(\sigma_\ast^m(f))]
    = \sum_{k=0}^{n-1} \sigma_\ast^k(h)
    = H.
    \end{align*}
    Thus $\mathrm{Im}(\sigma_\ast^n - \mathrm{id}) \cap C_c(X, \mathbb{N}) \neq \lbrace 0 \rbrace$. Since $n$ is arbitrary, then $(3)$ is false. Therefore, $(3) \Rightarrow (4)$.
\end{proof}

\section{Examples}
\label{section:examples}

In this section we show that Theorem \ref{thm:main} generalises three known results in the literature that characterise the AF embeddability of graph algebras, topological graph algebras and the crossed product $C(X) \rtimes_\sigma \mathbb{Z}$.

\subsection{Graph algebras}

Here we prove Schafhauser's result on graph algebras \cite{Schafhauser-AFEgraph} as a corollary of Theorem \ref{thm:main}.

A \emph{(directed) graph} is a tuple $E = (E^0, E^1, r, s)$ where $E^0$ and $E^1$ are countable sets equipped with the \emph{range} and \emph{source} maps $r, s: E^1 \rightarrow E^0$. The elements of $E^0$ are called \emph{vertices}, while $E^1$ contains the edges of the graph. We say that $E$ is \emph{row-finite} if $r^{-1}(v)$ is finite for every vertex $v$. Given a vertex $w$, we say that $w$ is a \emph{source} if $r^{-1}(w) = \emptyset$, and call $w$ a \emph{sink} if $s^{-1}(w) = \emptyset$. A finite or infinite sequence of edges $\mu_1 \mu_2 \dots$ is a \emph{path} if, for every two consecutive edges, we have $s(\mu_i) = r(\mu_{i+1})$. $E^\ast$, $E^\infty$  denote the sets of finite and infinite paths, respectively, while $E^n$ is the set of paths of \emph{length $n$}. 
The range of a path is the range of its first edge, while the source of a finite path is the source of its last element. We denote the length of a path $\mu$ by $\vert \mu \vert$. A vertex $v$ is considered path of length zero and $r(v) = s(v) = v$. Given two paths $\mu, \nu$ such that $s(\mu) = r(\nu)$, $\mu \nu$ is the path obtained by concatenating the two sequences. A finite path $\mu$ is a \emph{cycle} if $r(\mu) = s(\mu)$. This cycle has an \emph{entrance} if there exists an edge $e$ with  $r(e) = r(\mu_i)$ and $e \neq \mu_i$ for some $i$. For an introduction to graph algebras, see the book \cite{Raeburn} by Raeburn.

Given a row-finite graph with no sources, \cite[Theorem 4.2]{KPRR} gives an isomorphism from the graph algebra $C^*(E)$ to the C*-algebra $C^*(G_E)$ defined by $1_e \mapsto 1_{Z(e, r(e))}$, where
\begin{align*}
    G_E = \lbrace (x_0z, \vert x_0 \vert - \vert y_0 \vert, y_0 z) : x_0, y_0 \in E^\ast, z \in E^\infty \rbrace,
\end{align*}
and $Z(e, s(e)) = \lbrace (ex, 1, x) : x \in E^\infty, r(x) = s(e) \rbrace$. The set $G_E$ is an example of a Deaconu-Renault groupoid, where the map $\sigma: E^\infty \rightarrow E^\infty$ is given by 
\begin{equation*}
    \sigma(x_1 x_2 x_3 \dots) = x_2 x_3 \dots
\end{equation*}
and the topology on $E^\infty$ is generated by the cylinders $Z(\mu)$ for $\mu \in E^*$, where each $Z(\mu)$ is the set of all infinite sequences starting with $\mu$.

The groupoid is almost identical to the groupoid in \cite[Definition 2.3]{KPRR}, but we made some changes to make it compatible with the notation in Raeburn's book \cite{Raeburn}. In constrast to our definition, in \cite{KPRR} two consecutive edges $x_i, x_{i+1}$ satisfy $r(x_i) = s(x_{i+1})$. Also, their groupoid has elements of the form $(x, -k, y)$ instead of $(x, k, y)$.

We want to apply the isomorphism $C^*(E) \cong C^*(G_E)$, so we fix $E$ to be a row-finite graph with no sources. Note also that the map $\sigma$ is surjective if, and only if, $E$ has no sinks. In order to apply Theorem 1, we need $\sigma$ to be surjective, so we also assume throughout this subsection that $E$ has no sinks.

Schafhauser proved in \cite{Schafhauser-AFEgraph} that the graph algebra $C^*(E)$ is AF embeddable if, and only if, no cycle in $E$ has an entrance. In the next lemmas we study properties of $E^\infty$ that will help us to understand when Schafhauser's condition is false.

\begin{lemma}
\label{graph:lemma:np}
Suppose that no cycle in $E$ has an entrance. Given $x \in E^\infty$, then either
\begin{enumerate}
    \item all edges of $x$ are distinct, or
    \item there are unique $n \geq 0, p \geq 1$ such that $\sigma^{n+p}(x) = \sigma^n(x)$ and the elements $x_1, \dots,  x_{n+p}$ are distinct.
\end{enumerate}
\end{lemma}
\begin{proof}
Let $x \in E^\infty$. If all edges of $x$ are distinct, it is straightforward that $(2)$ is false. Now suppose that $(1)$ is false. This means that $x$ has repeated edges. Let $n$ be the smallest number such that $x_{n+1}$ appears more than once in $x$. Let $p$ be the smallest number such that $p \geq 1$ and $x_{n+p+1} = x_{n+1}$. Then $x_1, \dots, x_{n+p}$ are distinct. Let $\mu = x_1 \dots x_n$ and $\nu = x_{n+1} \dots x_{n+p}$. Note that $\nu$ is a cycle. By hypothesis, $\nu$ has no entrance. Then $\nu \nu \nu \dots$ is the only infinite path in $E$ with range $r(\nu)$. This implies that $x = \mu \nu \nu \dots$ and therefore $\sigma^{n+p}(x) = \sigma^n(x)$.
\end{proof}

\begin{lemma}
    Suppose that no cycle in $E$ has an entrance. Let $x \in E^\infty$ be such that condition $(2)$ of Lemma \ref{graph:lemma:np} holds for $n, p \geq 1$ (note that $n \geq 0$ is assumed). Given $y \in E^\infty$ with $\sigma(y) = x$, then it satisfies condition $(2)$ of that lemma for $n + 1$ and $p$.
\end{lemma}
\begin{proof}
As in the proof of Lemma \ref{graph:lemma:np}, we can write $x$ as $x = \mu \nu \nu \nu \dots$ where the finite paths $\mu$ and $\nu$ have lengths $n$ and $p$, respectively, and the edges of $\mu \nu$ are distinct. By hypothesis, we can write $y = e \mu \nu \nu \dots$, where $e \in E^1$. If $e$ is an edge in $\mu$, then $\mu$ contains a cycle with an entrance, which is a contradiction. If $e$ is in $\nu$, then $\mu_1$ is an entrance of $\nu$, which is a contradiction. Therefore $y$ satisfies condition $(2)$ for $n+1$ and $p$.
\end{proof}

Using similar arguments, we can prove the lemma below.

\begin{lemma}
\label{graph:lemma:yequalsx}
    Suppose that no cycle in $E$ has an entrance.  Let $x \in E^\infty$ be such that condition $(2)$ of Lemma \ref{graph:lemma:np} holds for $n = 0, p \geq 1$. If $\sigma^p(y) = x$ and $y$ satisfies condition $(2)$ of the same lemma for $n = 0$, then $y = x$.
\end{lemma}

The following lemmas study functions in $C_c(E^\infty, \mathbb{Z})$.

\begin{lemma}
    \label{graph:lemma:Fsigmakzero}
    Suppose that no cycle in $E$ has an entrance, and let $x \in E^\infty$ have distinct edges. Given $F \in C_c(E^\infty, \mathbb{Z})$, there exists $n \geq 1$ such that $F(\sigma^k(x)) = 0$ for all $k \geq n$.
\end{lemma}
\begin{proof}
    Let $ x = x_1 x_2 x_3 \dots$. Then the vertices $r(x_1), r(x_2), \dots$ are distinct, otherwise $E$ would have a cycle with an entrance. Moreover, $\sigma^k(x) \in Z(r(x_k))$ for $k \geq 0$. Note that $\lbrace Z(v) \rbrace_{v \in E^0}$ is a cover of $\mathrm{supp\hphantom{.}} F$ by open sets. By compactness, the set  $\mathrm{supp\hphantom{.}} F$ has a finite subcover $\lbrace Z(v) \rbrace_{v \in I}$. Therefore, there are finitely many elements $k$ such that $F(\sigma^k(x)) \neq 0$.
\end{proof}

By a similar argument, we can prove the following:
\begin{lemma}
\label{graph:lemma:rxdistinct}
    Suppose that no cycle in $E$ has an entrance. Let $\lbrace x^{(k)} \rbrace_{k \geq 1}$ be a sequence in $E^\infty$ such that
    \begin{itemize}
        \item $\sigma(x^{(k+1)}) = x^{(k)}$,
        \item $r(x^{(1)}), r(x^{(2)}), \dots$ are distinct.
    \end{itemize}
    Then there exists an $n \geq 1$ such that $F(x^{(k)}) = 0$ for all $k \geq n$.
\end{lemma}

\begin{lemma}
\label{graph:lemma:sigmamu}
    Let $\mu \in E^*$. Then
    \begin{align*}
        \sigma_\ast^{\vert \mu \vert + 1} (1_{Z(\mu)})
        = \sum_{\substack{e \in E^1 \\ r(e) = s(\mu)}} 1_{Z(s(e))}.
    \end{align*}
\end{lemma}
\begin{proof}
    Let $x \in E^\infty$. Then
    \begin{align*}
    \sigma_\ast^{\vert \mu \vert + 1}(1_{Z(\mu)})(x)
    = \sum_{y : \sigma^{\vert \mu \vert + 1}(y) = x} 1_{Z(\mu)}(y).
    \end{align*}
    If $1_{Z(\mu)}(y) = 1$, we must have $y = \mu e x$ for some $e \in E^1$. Then
    \begin{align*}
    \sigma_\ast^{\vert \mu \vert + 1}(1_{Z(\mu)}(x)
    = \sum_{\substack{e \in E^1 \\ r(e) = s(\mu) \\ s(e) = r(x)}} 1
    = \sum_{\substack{e \in E^1 \\ r(e) = s(\mu) }} 1_{Z(s(e))}(x).
    \end{align*}
    Therefore the result holds.
\end{proof}

Now we show Schafhauser's theorem \cite{Schafhauser-AFEgraph} as an application of Theorem \ref{thm:main} with the further assumptions that the graph is row-finite with no sinks or sources.

\begin{corollary}
The following are equivalent:
\begin{enumerate}[(i)]
\item $C^*(E)$ is AF embeddable,
\item $C^*(E)$ is AF quasidiagonal,
\item $C^*(E)$ is AF stably finite,
\item $C^*(E)$ is finite,
\item No cycle in $E$ has an entrance.
\end{enumerate}
\end{corollary}
\begin{proof}
Note that $C^*(E) \cong C^*(G_E)$. We are going to prove that $C^*(E)$ is not AF embeddable if, and only if, $E$ has a cycle with an entrance. First we prove the implication $(\Rightarrow)$. So, suppose that $C^*(E)$ is not AF embeddable. By Theorem \ref{thm:main}, there are two non-zero functions $f \in C_c(E^\infty, \mathbb{Z})$, $h \in C_c(E^\infty, \mathbb{N})$ such that $\sigma_\ast(f) - f = h$.

Assume that no cycle in $E$ has an entrance. With this assumption, we will find a contradiction by showing that $h = 0$. In order to do this, we will consider the different types of infinite paths $x$ given by Lemma \ref{graph:lemma:np}:
\begin{enumerate}[(a)]
    \item If $x$ has distinct edges, then $h(x) = 0$

    Let $x \in E^\infty$ have distinct edges. Suppose that $f(x) > 0$. Since $\sigma_\ast(f)(x) \geq f(x)$, there exists $x^{(2)} \in E^\infty$ with distinct edges such that $\sigma(x^{(2)}) = x$ and $f(x^{(2)}) > 0$. By repeating this argument, we find a sequence of infinite paths $x^{(k)}$ such that $x^{(1)} = x$ and, for all $k \geq 1$, $\sigma(x^{(k+1)}) = x^{(k)}$ and $f(x^{(k)}) > 0$. Since $x$ has distinct edges, Lemma \ref{graph:lemma:rxdistinct} implies that $f(x^{(k)}) = 0$ for some $k$, which is a contradiction. Therefore, we must have $f(x) \leq 0$. In particular, Lemma \ref{graph:lemma:np} implies that $y$ has distinct edges for all $y$ with $\sigma(y) = x$. Then $f(y) \leq 0$.

    Now suppose that $f(x) < 0$. Then
    \begin{align*}
        f(\sigma(x)) \leq \sigma_\ast(f)(\sigma(x))
        = \sum_{y : \sigma(y) = \sigma(x)} f(y) \leq f(x),
    \end{align*}
    because $\sigma(x)$ has distinct edges and thus $f(y) \leq 0$ for all $y$ with $\sigma(y) = \sigma(x)$. By repeating this argument, we have that, for all $k \geq 1$, $f(\sigma^k(x)) \leq f(x) < 0$. This contradicts Lemma \ref{graph:lemma:Fsigmakzero}. Therefore, $f(x) = 0$ for all $x \in E^\infty$ with distinct edges. Since every $y$ with $\sigma(y) = x$ also has distinct edges, it follows that $\sigma_\ast(f)(x) = 0$. Therefore $h(x) = 0$.

    \item If $x$ satisfies condition $(2)$ of Lemma \ref{graph:lemma:np} for $n, p \geq 1$, then $f(x) \leq 0$

    Assume that $f(x) > 0$. By similar arguments to the previous item, we can find a sequence $x^{(k)}$ such that $x^{(1)} = x$, $\sigma(x^{(k+1)}) = x^{(k)}$, and $f(x^{(k)}) > 0$. By Lemma \ref{graph:lemma:rxdistinct} all $r(x^{(k)})$ are distinct. This implies that $\lbrace Z(r(x^{(k)})) \rbrace_{k \geq 1}$ is an infinite open cover of $\mathrm{supp}\hphantom{.}f$ by disjoint subsets, which is a contradiction. Therefore $f(x) \leq 0$.

    \item If $x$ satisfies condition $(2)$ of of Lemma \ref{graph:lemma:np} for $n = 0, p \geq 1$, then $h(x) = 0$ and $\sigma_\ast^p(f)(x) = f(x)$

    Note that by assumption, $\sigma^p(x) = x$. By Lemma \ref{graph:lemma:yequalsx},
    \begin{align*}
        \sigma_\ast^p(f)(x)
        = \sum_{y : \sigma^p(y) = x} f(y)
        = \sum_{\substack{y : \sigma^p(y) = x \\ y \neq x}} f(y) + f(x) 
        \leq f(x).
    \end{align*}

    Using the same telescoping sums from the proof of Corollary \ref{corollary:ofthm1}, we have
    \begin{align*}
        f(x) \leq \sigma_\ast(f)(x)
        \leq \sigma_\ast^2(f)(x) 
        \leq \sigma_\ast^p(f)(x)
        \leq f(x).
    \end{align*}
    This implies that $h(x) = \sigma_\ast(f)(x) - f(x) = 0$.

    \item If $x$ satisfies condition $(2)$ of Lemma \ref{graph:lemma:np} for $n = 1$, $p \geq 1$, then $f(x) = 0$

    Let $z_0 = \sigma^n(x)$ and, for every $k \geq 1$, let $z_k = \sigma^{kn(2p-1)}(z_0)$. Note that $\sigma^p(z_0) = z_0$. Moreover, $z_0$ satisfies condition $(2)$ of Lemma \ref{graph:lemma:np} for $0$ and $p$. Note that $\sigma^n(z_{k+1}) = z_k$. Indeed,
    \begin{align*}
        \sigma^n(z_{k+1})
        = \sigma^{n}(\sigma^{(k+1)n(2p-1)}(z_0))
        = \sigma^{kn(2p-1)}(\sigma^{2pn}(z_0))
        = \sigma^{kn(2p-1)}(z_0)
        = z_k.
    \end{align*}
    Note that each $z_k$ satisfies condition $(2)$ of Lemma \ref{graph:lemma:np} for $0$ and $p$. So, we have
    \begin{align}
    \label{graph:eqn:fzk}
        f(z_k)
    \leq \sigma_\ast^n(f)(z_k)
    = \sum_{y: \sigma^n(y) = z_k} f(y)
    = f(z_{k+1}) + \sum_{\substack{\sigma^n(y) = z_k \\ y \neq z_{k+1}}} f(y).
    \end{align}
    But if $y \neq z_{k+1}$ and $\sigma^n(y) = z_k$, then $y$ satisfies condition $(2)$ of Lemma \ref{graph:lemma:np} for $n, p \geq 1$. So $f(y) \leq 0$ by item (b) of this proof. Hence, $f(z_k) \leq f(z_{k+1})$. It is straightforward from the definition that $z_p = z_0$.  This implies that $f(z_0) = f(z_1) = \dots = f(z_p)$. By \eqref{graph:eqn:fzk}, $f(y) = 0$ for all $y \neq z_1$ such that $\sigma^n(y) = z_0 = \sigma^n(x)$. Since $z_1$ satisfies condition $(2)$ of \ref{graph:lemma:np} for $0$ and $p$, we have that $x \neq z_1$. Therefore $f(x) = 0$.

    \item For every $x \in E^\infty$ with repeated edges, we have $h(x) = 0$

    Note that $x$ satisfies condition $(2)$ of Lemma \ref{graph:lemma:np} for $n \geq 0$ and $p \geq 1$. Also, $f(x) \leq 0$ by (b). Since we covered the case $n = 0$ in item $(c)$, we assume here that $n \geq 1$. We claim that $\sigma_\ast^k(f)(x) = 0$. If this were false, we could find a sequence $x^{(k)}$ such that $x^{(1)}= x$, $\sigma(x^{(k+1)}) = x^{(k)}$ and $f(x^{(k)}) \neq 0$, which would contradict the fact that $f$ is compactly supported. Hence,
    \begin{align*}
        0 = \sigma_\ast^k(f)(x) \geq \sigma_\ast^{k-1}(f)(x) \geq \dots \sigma_\ast(f)(x) \geq f(x) \leq 0.
    \end{align*}
    Therefore, $h(x) = \sigma_\ast(f)(x) - f(x) = 0$.
\end{enumerate}

Now let $x \in E^\infty$. In item (a), we showed that $h(x) = 0$ if $x$ has distinct edges; in item (e), we proved that $h(x) = 0$ if $x$ has repeated edges. This implies that $h = 0$, which is a contradiction. Therefore, $E$ has a cycle with an entrance.

Finally, we show $(\Leftarrow)$. Suppose $E$ has a cycle with an entrance. Then there is a finite path $\mu \in E^*$ with $s(\mu) = r(\mu)$, and there exists an $e_1 \in E^1$ such that $r(e_1) = s(\mu)$ and $e_1 \neq \mu_1$. Since $E$ has no sinks, all cycles in $E$ must have length greater than zero.

Note that $0 \leq \sigma_\ast(1_{Z(\mu)}) \leq 1_{Z(s(\mu_1))}$. In fact, $\sigma_\ast(1_{Z(\mu)}) \geq 0$ by definition of $\sigma_\ast$. Let $x \in E^\infty$. Then
\begin{align*}
    \sigma_\ast(1_{Z(\mu)})(x)
    = \sum_{y : \sigma(y) = x} 1_{Z(\mu)}(y)
    = 1_{Z(\mu)}(\mu_1 x) 1_{Z(s(\mu_1))}(x)
    \leq 1_{Z(s(\mu))}( x).
\end{align*}
By Lemma \ref{graph:lemma:sigmamu}, we have
\begin{align*}
    \sigma_\ast^{\vert \mu \vert + 1}(1_{Z(\mu)})
    = \sum_{\substack{e \in E^1 \\ r(e) = s(\mu)}} 1_{Z(s(e))}
    \geq  1_{Z(s(\mu_1))} + 1_{Z(s(e_1))}
    \geq  \sigma_\ast(1_{Z(\mu)}) + 1_{Z(s(e_1))}.
\end{align*}
Then there exists a non-zero function $h \in C_c(E^\infty, \mathbb{N})$ such that
\begin{align*}
    \sigma_\ast^{\vert \mu \vert}(\sigma_\ast(1_{Z(\mu)})) - \sigma_\ast(1_{Z(\mu)}) = h.    
\end{align*}
By Corollary \ref{corollary:ofthm1}, this implies that $C^*(E) \cong C^*(G_E)$ is not AF embeddable.

So far we have proved $(i) \Leftrightarrow (v)$. By Theorem \ref{thm:main}, we have the equivalence $(i) \Leftrightarrow (ii) \Leftrightarrow (iii) \Leftrightarrow (v)$. The implication $(iii) \Rightarrow (iv)$ is straightforward, and $(iv) \Rightarrow (v)$ follows from the same argument used in \cite[Theorem 1.1]{Schafhauser-AFEgraph}.
\end{proof}

\subsection{The crossed product $C(X) \rtimes_\sigma \mathbb{Z}$}

Now we apply Theorem \ref{thm:main} to show that it generalises Pimsner's theorem \cite[Theorem 9]{Pimsner} (see also \cite[Theorem 11.5]{Brown-quasidiagonal} by Brown). In contrast to these references, we do not assume that $X$ is a compact metric space. Instead, we fix $X$ to be a locally compact, Hausdorff, second countable, and totally disconnected space.

Note that, when $\sigma: X \rightarrow X$ is a homeomorphism, it follows from Deaconu's paper \cite{Deaconu} that $C(X) \rtimes_\sigma \mathbb{Z}$ is isomorphic to $C^*(\mathcal{G})$.

\begin{definition}
    We say that $\sigma$ \emph{compresses} an open set $\mathcal{U} \subset X$ if $\sigma(\mathcal{U}) \subsetneq \mathcal{U}$. So, the homeomorphism $\sigma$ \emph{compresses no compact open sets} if the following condition holds for all $\mathcal{U} \subset X$ compact open: $\sigma(\mathcal{U}) \subset \mathcal{U} \Rightarrow \sigma(\mathcal{U}) = \mathcal{U}.$
\end{definition}

The definition above is stated differently in \cite{Brown-quasidiagonal}. Recall that our topological space is more specific, being totally disconnected. For our purposes, both definitions are equivalent when $X$ is assumed to be compact. Now we show that Theorem \ref{thm:main} can be applied to prove \cite[Theorem 9]{Pimsner} (or \cite[Theorem 11.5]{Brown-quasidiagonal}).

\begin{corollary}
    Let $X$ be a locally compact, Hausdorff, second countable, totally disconnected space. Then the following are equivalent:
    \begin{enumerate}[(i)]
    \item $C_0(X) \rtimes_\sigma \mathbb{Z}$ is AF embeddable,
    \item $C_0(X) \rtimes_\sigma \mathbb{Z}$ is quasidiagonal,
    \item $C_0(X) \rtimes_\sigma \mathbb{Z}$ is stably finite,
    \item $\sigma$ compresses no compact open sets.
    \end{enumerate}
\end{corollary}
\begin{proof}
    Conditions (i)-(iii) are equivalent by Theorem \ref{thm:main}. We show the equivalence $(iv) \Leftrightarrow (i)$. Suppose $C_0(X) \rtimes_\sigma \mathbb{Z}$ is not AF embeddable. By Theorem \ref{thm:main}, there are non-zero functions $f \in C_c(X, \mathbb{Z})$ and $h \in C_c(X, \mathbb{N})$ such that
    \begin{align}
    \label{eqn:topgraph:notAFE}
    \sigma_\ast(f) - f = h.
    \end{align}

    Since $\sigma$ is a homeomorphism, we have $\sigma_\ast(f) = f \circ \sigma^{-1}$. By applying $\sigma$ to both sides of \eqref{eqn:topgraph:notAFE}, we have 
    \begin{align}
    \label{eqn:topgraph:notAFE2}
    f - f \circ \sigma = h \circ \sigma.
    \end{align}
    Note that $h \circ \sigma \in C_c(X, \mathbb{N})$ is non-zero, and $f \geq f \circ \sigma$.
    
    Given a natural number $n$, define the function $H_n = \sum_{i=0}^n h \circ \sigma^i$. Then $H_n$ assumes non-negative values. Similarly, define $H(x)$ for every $x \in X$ by
    \begin{align*}
        H(x) = \sum_{k=0}^\infty h \circ \sigma^k(x).
    \end{align*}
    It follows from \eqref{eqn:topgraph:notAFE2} that $H(x)$ is bounded. The dominated convergence theorem implies that the equation above defines a continuous function $H: X \rightarrow \mathbb{N}$. Also, $H$ is non-zero.

    We are going to find a set that is a compressed by $\sigma$. Since $H$ is non-zero and bounded, there exists an $x_0 \in X$ be such that $H(x_0) = h(x_0) \geq 1$. We claim that $f(x_0) \leq 0$. Suppose this is false. Then define the compact sets
    \begin{align*}
    A = \lbrace x \in X : f(x) \geq 1 \hspace{7pt}\text{and}\hspace{7pt} H(x) \geq 1 \rbrace
    \hspace{5pt}\text{and}\hspace{5pt}
     K_n = \lbrace x \in X : f(x) \geq 1 \hspace{7pt}\text{and}\hspace{7pt} H_{n-1}(x) = 0 \rbrace.
    \end{align*}
    Both sets are non-empty because $\sigma^{-n}(x_0) \in K_n \cap A$ for $n \geq 1$. By the finite intersection property, there exists an $x_\infty \in \cap_{n=1}^\infty (A \cap K_n)$, which implies that $H(x_\infty) = 0$ because $x_\infty$ is in the intersection of the sets $K_n$. But we also have $H(x_\infty) \geq 1$ because $x_\infty \in A$, which is a contradiction. Therefore, $f(x_\infty) \leq 0$.

    Define the sets 
    \begin{align*}
        \mathcal{U}
        &= \lbrace x : H(x) = h(x) = h(x_0) > 0 \hspace{7pt}\text{and}\hspace{7pt}
        f(x) = f(x_0) \rbrace, \hspace{7pt}\text{and} \\
        V
        &= \lbrace x : H(x) = 0
        \hspace{7pt}\text{and}\hspace{7pt}
        f(x) < f(x_0) \rbrace.
    \end{align*}
    Then $\sigma(\mathcal{U}) \cup \sigma(V) \subset V$, which implies that $\sigma$ compresses the compact set $\mathcal{U} \cup V$.

    Conversely, suppose that $\sigma$ compresses a compact open subset $\mathcal{U} \subset X$. Let $f = -1_\mathcal{U}$. Then, by the definition of $\sigma_\ast$,
    \begin{align*}
        \sigma_\ast(f) - f
        = - 1_{\mathcal{U}} \circ \sigma^{-1} + 1_{\mathcal{U}}
        = 1_{\mathcal{U}} - 1_{\sigma(\mathcal{U})}.
    \end{align*}
    Since $\mathcal{U} \subsetneq \sigma(\mathcal{U})$, the function $h = 1_{\mathcal{U} \setminus \sigma(\mathcal{U})} \in C_c(X, \mathbb{N})$ is nonzero and satisfies $\sigma_\ast(f) - f = h$. Therefore, $C_0(X) \rtimes_\sigma \mathbb{Z}$ is not AF embeddable.
\end{proof}

\subsection{Topological graph algebras}

In this subsection we prove that Theorem \ref{thm:main} is an application of Theorem 6.7 of \cite{Schafhauser-topologicalgraphs} by Schafhauser. We consider the definition of topological graphs introduced by Katsura in \cite{KatsuraI} and studied in the subsequent papers \cite{KatsuraII, KatsuraIII, KatsuraIV}.

\begin{definition}
    A \emph{topological graph} $E = (E^0, E^1, s, r)$ consists of two locally compact Hausdorff spaces $E^0, E^1$, and two maps $r, s: E^1 \rightarrow E^0$, where $s$ is a local homeomorphism and $r$ is continuous. A $v \in E^0$ is called a \emph{sink} if $s^{-1}(v) = \emptyset$. We say that $E$ is compact if $E^0, E^1$ are compact and $s$ is surjective.
\end{definition}

For consistence of notation, we  make a change of notation, using $s$ instead of $d$. Here we also assume that the spaces $E^0$ and $E^1$ are Hausdorff. Note that every directed graph is an example of topological graph. Now define $\varepsilon$-pseudopaths as in Schafhauser's paper.

\begin{definition}
    Let $E$ be a topological graph, and let $\varepsilon > 0$.  Let $d$ be a metric on $E^0$ compatible with the topology. An \emph{$\varepsilon$-pseudopath} in $E$ is a finite sequence $\alpha = (e_1, e_2, \dots, e_n) \in E^1$ such that for each $i = 1,\dots,n-1$, $d(s(e_i), r(e_{i+1})) < \varepsilon$. We write $r(\alpha) = r(e_1)$ and $s(\alpha) = s(e_n)$. An $\varepsilon$-pseudopath $\alpha$ is called an \emph{$\varepsilon$-pseudoloop} based at $r(\alpha)$ if $d(r(\alpha), s(\alpha)) < \varepsilon$.
\end{definition}

Here we make a slight change of notation by writing $\alpha$ as $(e_1, \dots, e_n)$ instead of $(e_n, \dots, e_1)$ from \cite{Schafhauser-topologicalgraphs}, to make our notation compatible with the notation of graphs. We also write that a $\varepsilon$-pseudoloop $\alpha$ is based at $r(\alpha)$ insted of $s(\alpha)$. Now we state the theorem that characterises the AF embeddability of topological graph algebras, writing $r$ instead $s$ in item (v), to agree with our notation.

\begin{theorem}
\label{thm:Schafhausertopological}
    \cite[Theorem 6.7]{Schafhauser-topologicalgraphs} if $E$ is a compact topological graph with no sinks, then the following are equivalent:
    \begin{enumerate}[(i)]
        \item $C^*(E)$ is AF embeddable,
        \item $C^*(E)$ is quasidiagonal,
        \item $C^*(E)$ is stably finite,
        \item $C^*(E)$ is finite,
        \item $r$ is injective and for every $v \in E^0$ and $\varepsilon > 0$, there exists an $\varepsilon$-pseudoloop in $E$ based at $v$.
    \end{enumerate}
\end{theorem}

Note that, for a Deaconu-Renault groupoid $\mathcal{G}$ corresponding to $\sigma: X \rightarrow X$, then $E = (E^0, E^1, s, r) \coloneqq (X, X, \sigma, \mathrm{id})$ is a topological graph  with $C^*(E) \cong C^*(\mathcal{G})$ by \cite[Proposition 10.9]{KatsuraII}. For Deaconu-Renault groupoids, we prove the equivalence of items $(i), (ii), (iii)$ and $(v)$ in Theorem \ref{thm:Schafhausertopological}.

\begin{corollary}
    Let $E =(E^0, E^1, s, r) \coloneqq (X, X, \sigma, \mathrm{id})$ be the topological graph corresponding to $\mathcal{G}$. Then the following are equivalent:
    \begin{enumerate}[(i)]
    \item $C^*(\mathcal{G})$ is AF embeddable,
    \item $C^*(\mathcal{G})$ is quasidiagonal,
    \item $C^*(\mathcal{G})$ is stably finite,
    \item $\mathrm{Im}(\sigma_\ast - \mathrm{id}) \cap C_c(X, \mathbb{N}) = \lbrace 0 \rbrace.$
    \end{enumerate}
    Moreover, consider the following statement:
    \begin{enumerate}[(i)]
        \setcounter{enumi}{4}
        \item for every $x \in X$ and $\varepsilon > 0$, there is an $\varepsilon$-pseudoloop in $E$ based at $x$.
    \end{enumerate}
    Then $(v) \Rightarrow (iv)$. If $X$ is compact, then all the items above are equivalent.
\end{corollary}
\begin{proof}
    Theorem \ref{thm:main} gives the equivalence $(i) \text{--} (iv)$. First we show $(v) \Rightarrow (iv)$. Suppose $(iv)$ is false. Then there are nonzero functions $f \in C_c(X, \mathbb{Z})$ and $h \in C_c(X, \mathbb{N})$ such that $\sigma_\ast(f) - f = h$. Since $\sigma$ is a local homeomorphism on a totally disconnected space and since $f, \sigma_\ast(f) \in C_c(X, \mathbb{Z})$, then there are disjoint compact open sets $V_1, \dots, V_n$, there is a number $1 \leq m \leq n$, and a surjective function $\rho: \lbrace 1, \dots, n \rbrace \rightarrow \lbrace 1, \dots, m \rbrace$ such that
    \begin{itemize}
        \item $\sigma$ is injective on each $V_i$,
        \item $\sigma(V_i) = V_{\rho(i)}$,
        \item There are integers $a_1, \dots, a_n$ and nonzero integers $b_1, \dots, b_m$ such that
        \begin{align*}
            \sigma_\ast(f) = \sum_{i=1}^n a_i 1_{V_i}
            \hspace{7pt}\text{and}\hspace{7pt}
            f = \sum_{j=1}^m b_j 1_{V_j}.
        \end{align*}
    \end{itemize}
    It is straightforward that, for every $j = 1, \dots, m$, $b_j = \sum_{i : \rho(i) = j} a_i$. We claim that $n > m$. Suppos this is false. Then $\rho$ is a bijection with $a_j = b_{\rho^{-1}(j)}$. Since $h$ is nonzero, then there exists an $l$ such that $a_j > b_j$. This implies that
    \begin{align*}
        \sum_{i=1}^n a_i > \sum_{i=1}^n b_i \Rightarrow
        \sum_{i=1}^n b_{\rho^{-1}(i)} = \sum_{i=1}^n b_i > \sum_{i=1}^n b_i,
    \end{align*}
    which is a contradiction. Therefore, $n > m$.

    Note that $a_i > .0$ for all $i$. Let $x \in C_n$. Then, for every $l > 0$, we have $\sigma^l(x) \in V_1 \cup \dots \cup V_m$. Let $\varepsilon > 0$ be such that $B(x, \varepsilon) \subset V_n$. This implies that there is no $\varepsilon$-pseudoloop in $E$ based at $x$. Therefore $(v)$ is false.

    Now assume that $X$ is compact. We prove $(iv) \Rightarrow (v)$. Suppose $(v)$  is false. Then there are $x \in X$ and $\varepsilon > 0$ such that there exists no $\varepsilon$-pseudoloop  in $E$ based at $x$.  By continuity, there exists an open set $V \subset B(x, \varepsilon)$ containing $x$ such that $d(\sigma(u), \sigma(x)) < \varepsilon$ for all $u \in V$. Note that
    \begin{align}
    \label{eqn:topologicalgraph:noloop}
    \sigma^i(V) \cap V = \emptyset
    \hspace{7pt}\text{for all }i \geq 1,
    \end{align}
    otherwise $X$ would have an $\varepsilon$-pseudoloop based at $x$. Let $\Omega = \bigcup_{i=0}^\infty \sigma^i(V)$. Then $\Omega$ is compact open because $X$ is compact and totally disconnected. By \ref{eqn:topologicalgraph:noloop}, we have $\sigma(\Omega) \subset \Omega$ and $\sigma(\Omega) \cap V = \emptyset$. So $\sigma(\Omega) \subsetneq \Omega$. Let $f = 1_{\sigma(\Omega)}$, define $h = \sigma_\ast(f) - f$. Then, for all $y \in X$, we have
    \begin{align*}
        \sigma_\ast(f)(y)
        = \sum_{u : \sigma(u) = y} 1_{\sigma(\Omega)}(u)
        \geq 1_\Omega(y)
        \geq 1_{\sigma(\Omega)}(y)
        = f(y).
    \end{align*}
    This implies that $h \in C_c(X, \mathbb{N})$. Moreover,
    \begin{align*}
        h = \sigma_\ast(f) - f
        \geq 1_{\Omega}
        - 1_{\sigma(\Omega)}
        = 1_{\Omega \setminus \sigma(\Omega)}
        \neq 0.
    \end{align*}
    Therefore, $(iv)$ is false.
\end{proof}

\begin{remark}
    When $X$ is not compact, the implication $(iv) \Rightarrow (v)$ might not be true. For example, consider $X = \mathbb{Z}$ and define $\sigma: X \rightarrow X$ by $\sigma(n) = n = 1$. Then $C^*(\mathcal{G})$ is AF embeddable, but there are no $\varepsilon$-pseudoloops for $0 < \varepsilon \leq 1/2$.
\end{remark}

\begin{remark}
    Note that for graph algebras and for C*-algebras of compact topological graphs, the C*-algebra is AF embeddable if, and only if, it is finite. We do not know if this equivalence for C*-algebras of Deaconu-Renault groupoids. We believe this is an open question.
\end{remark}

\subsection{Cuntz-Pimsner algebras}

Cuntz-Pimsner algebras form a large class of C*-algebras that includes, among others, the C*-algebra associated with a Deaconu-Renault groupoid $\mathcal{G}$. By applying the techniques described in \cite{RRS}, one can show that $C^*(\mathcal{G})$ is a Cuntz-Pimsner algebra, for which Schafhauser characterised the AF embeddability in \cite[Theorem C]{SchafhauserJFA2015}. We refer to Chapter 8 of \cite{Raeburn} for an introduction to Cuntz-Pimsner algebras.

\begin{definition}
    Let $G$ be a locally compact, Hausdorff, second countable, \'etale groupoid, and let $c: G \rightarrow \mathbb{Z}$ be a continuous cocycle. We say that $c$ is \emph{unperforated} if, for every $g \in G$ such that $c(g) = n > 0$, there exist composable $g_1, g_2, \dots, g_n$ such that each $c(g_i) = 1$ and $g = g_1 g_2 \dots g_n$.
\end{definition}

\begin{lemma}
The canonical cocycle $c: \mathcal{G} \rightarrow \mathbb{Z}$ on the Deaconu-Renault groupoid $\mathcal{G}$ is unperforated.
\end{lemma}
\begin{proof}
Let $(x, k, y) \in \mathcal{G}$ such that $c(x,k,y) = k > 0$. Then there exists an $m$ such that $\sigma^m+k(x) = \sigma^m(y)$. Since $\sigma$ is surjective, we can choose $z \in X$ with $\sigma^{k-1}(z) = y$. This implies that
\[
\sigma^m+k(x)
= \sigma^m(y)
= \sigma^{m+k-1}(z).
\]
This implies that $(x, 1, z) \in \mathcal{G}$ and $c(x,1,z) = 1$. Moreover, for $i = 1, \dots, k-1$, we have $(\sigma^{i-1}(z), 1, \sigma^i(z)) \in \mathcal{G}$ and $c(\sigma^{i-1}(z), 1, \sigma^i(z)) = 1$, where $\sigma^0(z) = z$. Hence,
\[
(x,k,y)
= (x,1, z)
(z, 1, \sigma(z))
\dots
(\sigma^{i-3}(z), 1, \sigma^{i-2}(z))
(\sigma^{i-2}(z), 1, y).
\]
Therefore, $c$ is unperforated.
\end{proof}

We will apply the following theorem by Rennie, Robertson, and Sims \cite{RRS} to write $C^*(\mathcal{G})$ as a Cuntz-Pimsner algebra.

\begin{theorem}
\label{thm:fromRRS}
\cite[Lemma 4 and Theorem 11]{RRS}
Suppose $G$ is a locally compact, Hausdorff, second countable, \'etale groupoid, and that $c: G \rightarrow \mathbb{Z}$ is an unperforated continuous cocycle. Let $X(G)$ be the completion in $C_r^*(G)$ of the set
\[
\lbrace f \in C_c(G) : \mathrm{supp}\hphantom{.} f \subset c^{-1}(1) \rbrace.
\]
Then the inclusion $I_0: C_c(c^{-1}(0)) \rightarrow C_c(G)$ extends to an embedding $I_0: C_r^*(c^{-1}(0)) \rightarrow C_r^*(G)$ and the inclusion $I_1: C_c(c^{-1}(1)) \rightarrow C_c(G)$ extends to a linear map $I_1: X(G) \rightarrow C_r^*(G)$. The set $X(G)$ is a C*-correspondence over $C_r^*(G)$ with module actions
\[
a \cdot \xi = I_0(a) \xi\hspace{15pt}
\text{and}\hspace{15pt}
\xi \cdot a = \xi I_0(a)
\hspace{15pt}
\text{for $\xi \in X(G)$, $a \in C_r^*(c^{-1}(0))$.}
\]
Moreover, the pair $(I_1, I_0)$ is a Cuntz-Pimsner covariant representation of $X(G)$, and the integrated form $I_1 \times I_0$ is an isomorphism of the Cuntz-Pimsner algebra $\mathcal{O}_{C^*_r(c^{-1}(0))}(X(G))$ onto $C_r^*(G)$.
\end{theorem}

\begin{remark}
    We make a change in notation, writing $\mathcal{O}_{C^*_r(c^{-1}(0))}(X(G))$ instead of $\mathcal{O}_{X(G)}$ to maintain consistency with the notation used in Schafhauser's paper \cite{SchafhauserJFA2015}.
\end{remark}

Now we write $C^*(\mathcal{G})$ as a Cuntz-Pimsner algebra.

\begin{corollary}
    The C*-algebra $C^*(\mathcal{G})$ is isomorphic to $\mathcal{O}_{C^*(c^{-1}(0))}(X(\mathcal{G}))$.
\end{corollary}
\begin{proof}
Since the canonical cocyle $c: \mathcal{G} \rightarrow \mathbb{Z}$ is unperforated, Theorem \ref{thm:fromRRS} implies that $C^*_r(\mathcal{G})$ is isomorphic to $\mathcal{O}_{C_r^*(c^{-1}(0))}(X(\mathcal{G}))$. The groupoids $\mathcal{G}$ and $c^{-1}(0)$ are amenable, then $\mathcal{O}_{C_r^*(c^{-1}(0))}(X(\mathcal{G})) = \mathcal{O}_{C_*(c^{-1}(0))}(X(\mathcal{G}))$ and $C^*_r(\mathcal{G}) = C^*(\mathcal{G})$.
\end{proof}

\begin{proposition}
\cite[Proposition 4.3]{SchafhauserJFA2015}
Suppose $A$ and $B$ are C*-algebras and $H$ is a row-finite $A\text{--}B$ C*-correspondence. If $T$ is an $A$-Ferdholm operator, then $T\otimes 1_H$ is a $B$-Fredholm operator and the map $T \mapsto T \otimes 1_H$ induces a positive group homomorphism $[H]: K_0(A) \rightarrow K_0(B)$.
\end{proposition}

The following theorem by Schafhauser characterises the AF embeddability for a large class of Cuntz-Pimsner algebras.

\begin{theorem}\cite[Theorem C]{SchafhauserJFA2015}
    Suppose $A$ is an AF algebra and $H$ is a separable C*-correspondence over $A$. Then the following are equivalent:
    \begin{enumerate}
        \item $\mathcal{O}_A(H)$ is AF embeddable,
        \item $\mathcal{O}_A(H)$ is quasidiagonal,
        \item $\mathcal{O}_A(H)$ is stably finite.
    \end{enumerate}
    Moreover, if $H$ is row-finite and faithful, the above conditions are equivalent to
    \begin{enumerate}
        \setcounter{enumi}{3}
        \item if $x \in K_0(A)$ and $[H](x) \leq x$, then $[H](x) = x$.
    \end{enumerate}
\end{theorem}

Note that, when $X$ is totally disconnected, the C*-algebra $C^*(c^{-1}(0))$ is AF. Hence, for the C*-algebra $C^*(\mathcal{G}) = \mathcal{O}_{C^*(c^{-1}(0))}(X(\mathcal{G}))$, conditions (1)--(3) are equivalent. It is not clear to the author, whether $X(\mathcal{G})$ is row-finite or faithful.

\bibliographystyle{abbrv}
\bibliography{bibliography}{}

\end{document}